\documentclass[11pt]{amsart}
\usepackage{amssymb,amsthm,amsmath,epsfig,latexsym}
\usepackage{anyfontsize}
\usepackage[pagebackref]{hyperref}
\usepackage{color}
\usepackage[all]{xy}
\usepackage{calc,times,verbatim}
\usepackage{geometry} 
\newcommand{\Proj}{{\mathbb P}}
\newcommand{\K}{{\mathbb K}}

\newcommand{\A}{{\mathcal{A}}}

\newcommand{\coker}{\mathop{\rm coker}\nolimits}
\newtheorem{defn0}{Definition}[section]
\newtheorem{prop0}[defn0]{Proposition}
\newtheorem{quest0}[defn0]{Question}
\newtheorem{thm0}[defn0]{Theorem}
\newtheorem{lem0}[defn0]{Lemma}
\newtheorem{corollary0}[defn0]{Corollary}
\newtheorem{example0}[defn0]{Example}
\newtheorem{remark0}[defn0]{Remark}

\newtheorem{Theorem}{Theorem}[section]
\newtheorem{Lemma}[Theorem]{Lemma}
\newtheorem{Corollary}[Theorem]{Corollary}
\newtheorem{Proposition}[Theorem]{Proposition}
\newtheorem{Remark}[Theorem]{Remark}
\newtheorem{Example}[Theorem]{Example}

\newtheorem{Definition}[Theorem]{Definition}
\newtheorem{Question}[Theorem]{Question}

\newtheorem{Caveat}[Theorem]{Caveat}

\newcommand{\rar}{\rightarrow}
\newcommand{\lar}{\longrightarrow}
\newcommand{\llar}{-\kern-5pt-\kern-5pt\longrightarrow}

\def\fm{{\mathfrak m}}

\def\xx{{\mathbf x}}

\def\Ht{{\rm ht}\,}
\def\depth{{\rm depth}\,}

\def\rk{{\rm rank}\,}
\def\syz{\mbox{\rm Syz}}

\def\NN{\mathbb N}

\def\pp{{\mathbb P}}

\begin{document}

\date{\sc\today}

\title{Rose--Terao--Yuzvinsky theorem for reduced forms}

\author{Ricardo Burity, Zaqueu Ramos, Aron Simis and \c{S}tefan O. Toh\v{a}neanu}

	 \keywords{Jacobian ideal, arrangement of hypersurfaces, depth, reduction of an ideal. \\
		\indent The second author was partially supported by a grant from CNPq (Brazil)(304122/2022-0).\\
		\indent The third author was partially supported by a grant from CNPq (Brazil) (301131/2019-8).\\
		\indent Burity's address: Departamento de Matemática, Universidade Federal da Paraiba, 58051-900, J. Pessoa, PB,  Brazil,
		Email: ricardo@mat.ufpb.br\\
		\indent Ramos' address: Departamento de Matem\'atica, CCET, Universidade Federal de Sergipe, 49100-000 S\~ao Cristov\~ao, SE, Brazil
		Email: zaqueu@mat.ufs.br\\
        \indent Simis' address: Departamento de Matemática, Universidade Federal de Pernambuco, 50740-560 Recife, PE,  Brazil,
		Email: aron@dmat.ufpe.br\\
		\indent Toh\v{a}neanu's address: Department of Mathematics, University of Idaho, Moscow, Id 83844-1103, USA, 
		Email: tohaneanu@uidaho.edu}

\begin{abstract}
	Yuzvinsky and  Rose--Terao have shown that the homological dimension of the gradient ideal of the defining polynomial of a generic hyperplane arrangement is maximum possible.
	In this work one provides yet another proof of this result, which in addition is totally different from the one given by Burity--Simis--Toh\v{a}neanu.
Another main drive of the paper concerns a version of the above result in the case of a product of general forms of arbitrary degrees (in particular, transverse ones).
Finally,  some relevant cases of non general forms are also contemplated. 
\end{abstract}

\maketitle



\section{Introduction}

Let $R:=\mathbb K[x_1,\ldots,x_n]$ denote a standard graded polynomial ring over a field $\K$.
For any form ${f}\in R$, one has the module of {\em logarithmic derivations} of ${f}$ defined as $\mathrm{Derlog}(f):=\{\theta\in {\rm Der}(R)| \theta ({f})\subset \langle { f} \rangle\}.$ 
Its $R$-module structure  is well known if ${\rm char}(\mathbb K)$ does not divide $m=\mathrm{deg}({ f})$, namely:
\begin{equation*}
\mathrm{Derlog}(f)= \mathrm{Syz}(J_f)\oplus R\theta_E,
\end{equation*}
where $\mathrm{Syz}(J_{f})\subset R^n$ is up to the identification ${\rm Der}(R)=R^n$ the syzygy module of the  partial derivatives of  $f$ and $\theta_E=\sum_{i=1}^{n}x_i\frac{\partial}{\partial{x_i}}$ is the Euler derivation.

The above notation  brings upfront an object slightly behind the corner: the ideal of $R$ generated by the  partial derivatives of  ${f}$, the so-called {\em gradient ideal}  (or {\em Jacobian ideal}) $J_{f}$ of ${f}$.

The $\K$-algebra $R/J_{f}$ is often called the {\em Milnor algebra} of ${f}$. At times, this terminology seems to obfuscate the search of the invariants of $J_{f}$ itself, whereas the latter have been around way back (see, e.g., \cite{KoszulSim1977}; also \cite{JacobSim2005}, with an even earlier mention).
Quite a bit of  geometric and cohomological  uses of the gradient ideal, at least in the case where $V({f})$ is smooth, was developed in the work of Griffiths (\cite{Griffiths1969}) and Carlson--Griffiths (\cite{Carlson-Griffiths1979}), later, by Peters--Steenbrink (\cite{Peters-Steenbrink1983}), and this year's Aguilar--Green--Griffiths (\cite{AGG2024}).

Unfortunately, very often  the algebraic background around the gradient ideal would not be fully shown.
A major goal of this paper is to bring out some of this by focusing upon the related homological behavior. 
More specifically, we look at the depth of the ring $R/J_{{f}}$ under three different perspectives.

In the first, we consider a central hyperplane arrangement $\mathcal{A}$ of rank $n$ in affine space $\mathbb{K}^n$ over an infinite field $\mathbb{K}$ and let $l_1,\ldots, l_m\in R:=\mathbb K[x_1,\ldots,x_n]$ denote the linear forms defining the corresponding hyperplanes.
Set $F:=l_1\cdots l_m\in R$, the defining polynomial of $\mathcal{A}$. 
In the case where the arrangement is generic, -- meaning that every subset of $\{l_1,\ldots,l_m\}$ with $n$ elements is $\K$-linearly independent -- and that $m\geq n+1$, Rose-Terao (\cite{RoseTerao}) and Yuzvinsky (\cite{Yuzvinsky}) have established that the homological dimension of $\mathrm{Derlog}(\mathcal{A})$ is $n-2$.   The statement is equivalent to having  ${\rm depth}(R/J_{F})=0$, or still that the irrelevant maximal ideal $\fm:=\langle x_1,\ldots,x_n\rangle$ is an associated prime of $R/J_{F}$.
In their beautiful paper,  Rose and Terao took the approach of  establishing explicit free resolutions of the modules $\Omega^q(\mathcal{A})$  of logarithmic $q$-forms, for $0\leq q\leq n-1$, in the spirit of Lebelt's work (\cite{Lebelt}). Then, drawing on the isomorphism $\mathrm{Derlog}(\mathcal{A})\simeq \Omega^{n-1}(\mathcal{A})$, they derived the free resolution of the module of logarithmic derivations, and hence of $R/J_{F}$ as well.

In the second perspective, motivated by the first one, we  take a more general viewpoint, by dealing with the case of general forms, as a natural replacement for a generic arrangement.
Thus, let $\mathbb K$ be an infinite field, of characteristic zero or of characteristic not dividing certain critical integers. Let $R:=\K[x_1,\ldots,x_n]$ ($n\geq 2$) denote a standard graded polynomial ring over $\K$. 
We focus on {\em general} forms $f_1,\ldots,f_m\in R, m\geq 2$, of arbitrary degrees $\deg(f_i)=d_i\geq 1, i=1,\ldots,m$, and on the associated product $F:=f_1\cdots f_m$.

Finally, the third perspective is a natural complement to the previous perspective in that it deals with arbitrary forms,  emphasizing a small numbers of those.
As next-best properties avoiding genericity, we focus on transversality  and smoothness variously as supporting assumptions.

In all three perspectives, the main goal is to decide when  $\depth {R/J_F}=0$ (which is equivalent to showing that $J_F$ is non-saturated  with respect to ${\frak m}:=\langle x_1,\ldots,x_n\rangle$). 

Accordingly, we formally set:
\begin{Definition}\rm
A set of forms $\{f_1,\ldots,f_m\}\subset R=\K[x_1,\ldots,x_n]$ will be said to satisfy the RTY-{\em property} (RTY for Rose--Terao--Yuzvinsky) if $\depth R/J_F=0$, where $F=f_1\cdots f_m$. Equivalently, the property says that $R/J_F$ has homological dimension $n$.
\end{Definition}
When $n=3$, the environment  narrows down, in the sense that the failure of the RTY-property means that $F$ is a free divisor, hence their tight relationship in this dimension.
In higher dimensions no such conundrum persists, so that having $\depth R/J_F\geq 1$ may not tell much.
Yet, even in low dimensions, the RTY-property itself does not reveal much as to the homological classification (see, e.g., Subsection~\ref{Subsection4.2}).

Various aspects of this subject have been recently brought up in \cite{Buse_et_al2021}, which is an important addition to the questions treated in this work.  Their main focus is on the hypersurface $V(F)$ itself and the related  algebra $R/J_F$, with a particular quest for the Castelnuovo--Mumford regularity of the latter.
In the present work our emphasis -- or, perhaps we should say, our methods and results -- are different.

 Next is a brief description of the sections, along with the main results of the paper.
 
 Section~\ref{Section2} deals with distinct linear forms $l_1,\ldots,l_m\in R:=\K[x_1,\ldots,x_n]$ defining a central arrangement $\mathcal{A}$ in affine space $\K^n$ over an infinite field $\mathbb{K}$ with $m\geq n+1.$ 
We assume once for all that $\mathcal{A}$ is {\it generic}, that is, any $n$ among the defining linear forms are $\mathbb{K}$-linearly independent, a property which being invariant under change of coordinates allows to assume that  $l_i=x_i$ for every $1\leq i\leq n.$
Let $A$ denote the coefficient matrix of the remaining linear forms, i. e., one has a matrix equality
$$A\, [x_1 \, \cdots\, x_n]^t=[l_{n+1}\, \cdots\, l_m]^t.$$
Introduce yet another player, namely, the ideal $I \subset R$ generated by the $(m-1)$-fold products of the defining linear forms, and let $\Gamma$ denote its well-known syzygy matrix based on the assumed shape of the linear forms of $\mathcal{A}$. 
Finally, consider the following $(m-n)\times (m-1)$ product matrix
\begin{equation*}
	\mathfrak{S}: =	[
	\begin{array}{c|c}
		-A & \mathbb{I}_{m-n}
	\end{array}
	]\; \Gamma,
\end{equation*}
where $\mathbb{I}_{m-n}$ denotes the identity matrix of the ascribed size, while $[\begin{array}{c|c}
	-A & \mathbb{I}_{m-n} \end{array}] $ is a two block matrix.
The main technical result of the Section says that, if char$(\K)\nmid m$, then
$I_{m-n}(\mathfrak{S})=\langle x_1,\ldots,x_n\rangle^{m-n}$, the ideal of maximal minors of  $\mathfrak{S}$.
Based on this preliminary, the main results are Theorem~\ref{RTYoriginal} and Theorem~\ref{minimal_reduction}, respectively proving the original result of Rose--Terao and Yuzvinsky, and the result of \cite{BuSiTo2022} to the  effect that the gradient ideal $J_F$ is a  minimal reduction of the ideal $I$ of $(m-1)$-fold products.
We give entirely fresh proofs for both results, with the first drawing upon \cite[Theorem D]{AndSim1986}.

 Section~\ref{Section3} contains the complete discussion of the RTY-property for general forms.
 Our main result is the RTY-property for a set of $m$ forms of arbitrary degrees $\geq 2$ such any subset with $\leq n$ elements is general.
 In addition, we give the explicit shape of the minimal free resolution of the gradient ideal of the product of the forms.
 The Betti numbers of the resolution are accounted for, as well as some notion of the resolution maps. 
 
 The main results are Theorem~\ref{basic_open_set}, Proposition~\ref{codim_jac} and Theorem~\ref{main_general_forms}, the last of these drawing upon the main theorem of \cite{AndSim1981}.
 
 A driving result of the last section is Theorem~\ref{two__forms_one_smooth} which gives two  conditions for such a set of two forms to have RTY-property. These conditions are easy enough to check, allowing for subsequent corollaries and examples of various kinds, including one regarding so-called Fermat forms.
 In between we give full treatment to the case of two non-degenerate plane quadrics, giving their exact shape in order that their product be a free divisor. As a complement, we give the full homological classification picture of the product in terms of the Bourbaki degree introduced in \cite{Bour2024} and certain plane curves introduced by Dimca and Sticlaru.
 
 The main results are Theorem~\ref{two__forms_one_smooth} and Proposition~\ref{free-again}.

For all purposes in this work, the ring $R$ is assumed to be standard graded.
 In what follows $\syz (J_F)\subset R^n$ denotes the syzygy module of the partial derivatives of $F$, and ${\rm indeg}(\syz (J_F))$ is the least (standard) degree of a nonzero element of this graded module.
 
 Throughout,  the Jacobian matrix of a set of forms $\mathbf{f}=\{f_1,\ldots,f_r\}\subset R$  will be denoted $\Theta(\bf f)$.

\section{Linear forms}\label{Section2}
In this section we revisit the case of linear forms in generic arrangement and retrieve results of \cite{RoseTerao} and \cite{BuSiTo2022}.
The present approach is essentially matrix-theoretic and quite diverse.

\subsection{Coefficient matrix}

Let $l_1,\ldots,l_m\in R:=\K[x_1,\ldots,x_n]$ be distinct linear forms defining a rank $n$ central arrangement $\mathcal{A}$ in affine space $\K^n$ over an infinite field $\mathbb{K}$ with $m\geq n+1.$  
Quite generally here, the Jacobian matrix 
$\Theta=(\partial l_i/\partial x_j)$ coincides with the matrix of coefficients of the forms with respect to $\{x_1,\ldots,x_n\}$:
$$\Theta\left[\begin{matrix}x_1\\\vdots\\x_n\end{matrix}\right]=\left[\begin{matrix}l_1\\\vdots\\l_m\end{matrix}\right].$$
By the same token, for any sequence of indices $1\leq i_1< \cdots < i_n\leq m$, one has
$$\Theta_{i_1,\ldots,i_n}\left[\begin{matrix}x_1\\\vdots\\x_n\end{matrix}\right]=\left[\begin{matrix}l_{i_1}\\\vdots\\l_{i_n}\end{matrix}\right],$$
where $\Theta_{i_1,\ldots,i_n}$ denotes the $n\times n$ submatrix of $\Theta$ formed by the rows $i_1,\ldots,i_n.$
In order words, this matrix is the representative matrix of the set $\{l_{i_1},\ldots, l_{i_n}\}$ in the canonical basis $\{x_1,\ldots,x_n\}$. Thus,$\{l_{i_1},\ldots, l_{i_n}\}$ is linearly independent over $\K$ if and only if $\Theta_{i_1,\ldots,i_n}$ has non-vanishing determinant.

Suppose henceforth that the arrangement $\mathcal{A}$ is {\it generic}, that is, any $n$ among the defining linear forms are $\mathbb{K}$-linearly independent. 
This property is invariant under change of coordinates, thus we will assume that  $l_i=x_i$ for every $1\leq i\leq n.$ Then we have
$$\Theta=\left[\begin{matrix}\mathbb{I}_{n}\\ A\end{matrix}\right]$$ 
where $A=(\partial l_i/\partial x_j)_{n+1\leq i\leq m,\, 1\leq j\leq n }.$ In other words, $A$ is the coefficient matrix of the ``remaining" linear forms:
\begin{equation}\label{coefficient}
	A\left[\begin{matrix}x_1\\\vdots\\x_n\end{matrix}\right]= \left[\begin{matrix}l_{n+1}\\\vdots\\l_m\end{matrix}\right]
\end{equation}
By the above,  we have the well-known fact that $\mathcal{A}$ is a generic arranjement if and only if any $n$-minor of the full matrix $\Theta$ is non-vanishing. 
Actually, we claim more:

\begin{Lemma}\label{minorsofA}
	Every minor of $A$ is non-vanishing.
\end{Lemma}
\begin{proof} Given $1\leq r\leq \min\{m-n,n\}$, let $A'$ denote an $r\times r$ submatrix of $A$.
	Up to row and column permutations 
	we may suppose  that  $A'$ lies on the upper left corner of $A.$  Now, consider the following $n\times n$ submatrix of $\Theta$:
	$$B=\left[\begin{matrix}\boldsymbol0&\mathbb{I}_{n-r}\\A'&\ast\end{matrix}\right].$$
	Clearly, $\det B=-\det A'$, hence the result follows from the above observations.
\end{proof}

Throughout, let $F:=x_1\cdots x_{n}l_{n+1}\cdots l_m$ stand for the defining form of the arrangement $\mathcal{A}$.
We bring in the ideal $I \subset R$ generated
by the $(m-1)$-fold products of the defining forms.
Thus, $I=\langle f_1,\ldots,f_m\rangle,$ where $f_{i}:=F/x_i$ for every $1\leq i\leq n$ and $f_i:=F/l_i$ for every $n+1\leq i\leq m$.

It is well-known that $I$ is a codimension two perfect ideal admitting the following matrix as syzygy matrix relative to the above set of generators (see \cite{Sch2}):
\small{$$\Gamma=\left[\begin{matrix}
		-x_1&\\
		&\ddots\\
		&&-x_n\\
		&&&-l_{n+1}\\
		&&&&\ddots\\
		&&&&&-l_{m-1}\\
		l_m&\cdots&l_m&l_m&\cdots&l_m
	\end{matrix}\right],$$}
where empty slots have null entries.

The following product matrix of size $(m-n)\times (m-1)$ will play a decisive role in the argument that follows: 
\begin{equation*}
	\mathfrak{S}: =	[
	\begin{array}{c|c}
		-A & \mathbb{I}_{m-n}
	\end{array}
	]\; \Gamma
\end{equation*}
Here, $\mathbb{I}_{m-n}$ denotes the identity matrix of the ascribed size, while $[\begin{array}{c|c}
	-A & \mathbb{I}_{m-n}\end{array}] $ is a two block matrix.

\begin{Proposition}\label{minors_are_monomials}  	{\rm (char$(\K)\nmid m$)}\;
	$I_{m-n}(\mathfrak{S})=\langle x_1,\ldots,x_n\rangle^{m-n}.$ 
\end{Proposition}
\begin{proof}	
	Say,  $A=(a_{n+i,j})_{1\leq i\leq m-n,\,1\leq j\leq n}.$ Then,
	\begin{equation}\label{Sexplicita}
		\mathfrak{S} 
		=\left[\begin{matrix}
			a_{n+1,1}x_1&\cdots&a_{n+1,n}x_n&-l_{n+1}&&\\
			\vdots&\ddots&\vdots&&\ddots\\
			a_{m-1,1}x_{1}&\cdots&a_{m-1,n}x_n&&&-l_{m-1}\\
			a_{m,1}x_{1}+l_m&\cdots&a_{m,n}x_{n}+l_m&l_m&\cdots&l_m	
		\end{matrix}\right].
	\end{equation}
	By assumption, $1/m\in \K$. Let  $U$ be the transpose of the $1\times(m-1)$ matrix $\left[\begin{matrix}1/m&\cdots&1/m\end{matrix}\right].$ 
	Then,
	\begin{equation}\label{SU}
		\mathfrak{S}U=
		\left[\begin{matrix}
			\frac{1}{m}(-l_{n+1}+\sum_{j=1}^na_{n+1,j}x_j)\\
			\vdots\\
			\frac{1}{m}(-l_{m-1}+\sum_{j=1}^na_{m-1,j}x_j)\\
			\frac{1}{m}((m-1)l_m+\sum_{j=1}^na_{m,j}x_j)
		\end{matrix}\right]=\left[\begin{matrix}
			0\\
			\vdots\\
			0\\
			l_m
		\end{matrix}\right],
	\end{equation}
	since by \eqref{coefficient}, $l_{n+i}=\sum_{j=1}^na_{n+1,j}x_j$ for every $1\leq i\leq m-n.$ 
	
	Now, consider  the $(m-1)\times(m-1)$ matrix  $V:=\left[\begin{matrix} U&\cdots&U\end{matrix}\right]$ and the $(m-1)\times m$ two block matrix $W:=\left[\begin{array}{c|c}\mathbb{I}_{m-1}-V&-U\end{array}\right].$
	Clearly, $	\mathfrak{S}W=
	\left[\begin{array}{c|c}\mathfrak{S}-\mathfrak{S}V&-\mathfrak{S}U\end{array}\right]$.
	Drawing upon \eqref{Sexplicita} and \eqref{SU}, one gets
	\begin{equation*}
		\mathfrak{S}W=
		\left[\begin{array}{c|ccccccc}
			\underbrace{\begin{matrix}
					a_{n+1,1}x_1&\cdots&a_{n+1,n}x_n\\
					\vdots&\ddots&\vdots\\
					a_{m-1,1}x_{1}&\cdots&a_{m-1,n}x_n\\
					a_{m,1}x_{1}&\cdots&a_{m,n}x_{n}
			\end{matrix}}_{=:C}
			&
			\underbrace{\begin{matrix}
					-l_{n+1}&\\
					&\ddots\\
					&&-l_{m-1}\\
					&&&-l_m
			\end{matrix}}_{=:D}
		\end{array}\right].
	\end{equation*}

	\smallskip
	
	{\sc Claim.} The ideal of $(m-n)$-fold products of $\{x_1,\ldots,x_n,l_{n+1},\ldots,l_{m}\}$ is contained in $I_{m-n}(\mathfrak{S}W)$.
	
	Namely, we show that, given integers $0\leq r\leq \min\{m-n,n\}$ and $0\leq s\leq m-n$ such $r+s=m-n$ and subsets of indices $\{i_1,\ldots,i_r\}\subset\{1,\ldots,n\}$ and $\{j_1,\ldots,j_s\}\subset\{n+1,\ldots,m\}$, one has $x_{i_1} \cdots x_{i_r}l_{j_1}\cdots l_{j_s}\in \mathfrak{S}W$.
	
	Consider the $(m-n)$-minor of  $\mathfrak{S}W$ with column indices $i_1,\ldots,i_r$ in block $C$ and $j_1,\ldots,j_s$ in block $D$.
	
	Up to permutations of  rows and columns we can suppose  
	that it has the form
	$$\left[\begin{matrix}
		\ast&l_{j_1}&\cdots&0\\
		\vdots&\vdots&\ddots&\vdots\\
		\ast&0&\cdots&l_{j_s}\\
		\hline
		L_A&\boldsymbol{0}&\cdots&\boldsymbol{0}
	\end{matrix}\right]$$
	for a certain $r\times r$ submatrix $L_A$ of $D$. Let $\Delta:=
	l_{j_1}\cdots l_{j_s}\det L_A$ denote its determinant.
	Note that $L_A$, as is the case of any $r\times r$-minor of  $D$, has the form $x_{i_1}\cdots x_{i_r}\det \tilde{A}$ where $\tilde{A}$ is an $r\times r$ submatrix of $A.$
	By Lemma \ref{minorsofA}, $c:=\det \tilde{A}\neq 0$. 
	Therefore, $x_{i_1}\cdots x_{i_r}l_{j_1}\cdots l_{j_s}=c^{-1}\Delta\in I_{m-n}(\mathfrak{S}W).$
	
	This takes care of the claim.

	To conclude,  the ideal of $(m-n)$-folds products of $\{x_1,\ldots,x_n,l_{n+1},\ldots,l_{m}\}$ is the ideal $\langle x_1,\ldots,x_n\rangle^{m-n}$ (see, e.g., \cite[Lemma 2.5]{BuSiTo2022}, or, for a more general statement, \cite[Theorem 3.1]{To1}). Thus,  
	
	$$
	\langle x_1,\ldots,x_n\rangle^{m-n}\subset I_{m-n}(\mathfrak{S}W)\subset I_{m-n}(\mathfrak{S})\subset\langle x_1,\ldots,x_n\rangle^{m-n}$$
	Hence, $I_{m-n} (\mathfrak{S)}=\langle x_1,\ldots,x_n\rangle^{m-n}.$
\end{proof}

\subsection{Generic arrangement}

As an application, we provide yet another approach to proving the RTY-property of a generic arrangement of linear forms, giving  details of the respective minimal free resolution as in \cite[Theorem 4.5.3 and Corollary 4.5.4]{RoseTerao} and \cite[Theorem 2.9]{BuSiTo2022}.

Throughout, $R=\K[x_1,\ldots,x_n]$ denotes a polynomial ring over the field $\K$.
The following result is well-known, but we give a proof for the reader's convenience.

\begin{Lemma}\label{generator-minorsordered}
	Let $\phi$ be an $m\times (m-1)$ matrix over $R$ of $\K$-linear forms, and let $\Delta_i, 1\leq i\leq m,$ denote its ordered signed $(m-1)$-minors. Suppose that
	$I_{m-1}(\phi)$ has height $2.$  If $h_1,\ldots,h_m$ are homogeneous polynomials of $R$, of degree $m-1$, such that 
	$$\left[\begin{matrix} h_1&\cdots&h_m\end{matrix}\right]\phi=\boldsymbol0,$$
	then there is an element $c\in \K\setminus\{0\}$ such $h_i=c\Delta_i$ for every $1\leq i\leq m.$ 
\end{Lemma}
\begin{proof} Since $\left[\begin{matrix} \Delta_1&\cdots&\Delta_m\end{matrix} \right]\phi=\boldsymbol0$
holds as well, then $\left[\begin{matrix} h_1&\cdots&h_m\end{matrix}\right]$ and  $\left[\begin{matrix} \Delta_1&\cdots&\Delta_m\end{matrix}\right]$ are proportional over the fraction field of $R.$ Hence, there are $p,q\in R\setminus \{0\}$ with gcd$(p,q)=1$ such that $qh_i=p\Delta_i$ for every $1\leq i\leq m.$ Clearly, then $q$ divide $\Delta_i$ for every $1\leq i\leq m.$ But, since $I_{m-1}(\phi)$ has height $2$ it follows by the Hilbert-Burch theorem that $q\in \K.$ Hence, for reason of degrees, $p\in \K$ as well. Thus, we can take $c=p/q$ to conclude the proof.
\end{proof}

\begin{Theorem} \label{RTYoriginal}	{\rm (char$(\K)\nmid m$)} \;
	Let $F\in R=\K[x_1,\ldots,x_n]$ stand for the defining form of a generic hyperplane arrangement of size $m\geq n+1$, and let $J_F$ denote its gradient ideal. Then $R/J_F$ has homological dimension $n$, with minimal graded free resolution 
	\begin{eqnarray*}
			0\rightarrow R^{\beta_{n-1}}(-(2m-2)) & \rightarrow & R^{\beta_{n-2}}(-(2m-3))
		\rar  \cdots \rightarrow R^{\beta_1}(-(2m-n))\\ 
		&\rar & R^n(-(m-1))\rar R \rar R/J_F \rar 0,
		\end{eqnarray*}
		where $\beta_i={m-n+i-2\choose i-1 } {m-1\choose m-n+i}$, for $1\leq i \leq n-1$.
	
	\smallskip
	
	In particular, the defining linear forms of the arrangement satisfy the RTY-property.
\end{Theorem}
\begin{proof} 
	Note that
	\begin{equation*}
		\left[\begin{matrix}\frac{\partial F}{\partial x_1} &\cdots&\frac{\partial F}{\partial x_n}&f_{n+1}&\cdots f_m\end{matrix}\right]=\left[\begin{matrix}f_1&\cdots&f_m\end{matrix}\right]\underbrace{\left[\begin{matrix} \mathbb{I}_{n}&\boldsymbol0\\A&\mathbb{I}_{m-n}\end{matrix}\right]}_{=:\mathfrak{A}}.
	\end{equation*}

	Since the inverse of $\mathfrak{A}$ is  $\mathfrak{A}^{-1}=\left[\begin{matrix} \mathbb{I}_{n}&\boldsymbol0\\-A&\mathbb{I}_{m-n}\end{matrix}\right]$, then multiplying through the above relation on the right by the  $m\times (m-1)$ matrix $\mathfrak{A}^{-1} \Gamma$, it follows that 
	
	$$\left[\begin{matrix}\frac{\partial F}{\partial x_1} &\cdots&\frac{\partial F}{\partial x_n}&f_{n+1}&\cdots f_m\end{matrix}\right]\mathfrak{A}^{-1} \Gamma=\boldsymbol0.$$
	Since $I=I_{m-1}(\Gamma)=I_{m-1}(\mathfrak{A}^{-1} \Gamma)$, Lemma \ref{generator-minorsordered} implies the following

	{\sc Claim 1.} $J_F$ is  the ideal generated by the $(m-1)$-minors of $\mathfrak{A}^{-1} \Gamma$ fixing the last $(m-n)$-rows.
	
	\smallskip

Next, note that the lowest $(m-n)\times (m-1)$ submatrix of  $\mathfrak{A}^{-1} \Gamma$ is our previous matrix $\mathfrak{S}$. In fact, we can write
	
	$$\mathfrak{A}^{-1} \Gamma=\left[\begin{matrix}\mathfrak{S}'\\
		\mathfrak{S}\end{matrix}\right]$$
	where  $\mathfrak{S}'$ is an $n\times (m-1)$ matrix.

	\smallskip
	
	{\sc Claim 2.} $\Ht I_{m-n}(\mathfrak{S})= m-1-(m-n)+1=n$ (maximum possible).
	
	\smallskip	
	
	This is an immediate consequence of Proposition~\ref{minors_are_monomials}.
	
	By {\sc Claim 2},  $\coker \mathfrak{S} $ is as known  resolved by the Buchsbaum-Rim complex (see \cite[Corollary A2. 13]{Eis}):
	
	\begin{eqnarray}\label{B-R-complex}
		0\to F_{n-1} \stackrel{\partial}\lar \cdots \stackrel{\partial}\lar F_{1}\stackrel{\eta}\lar R(-1)^{m-1}\stackrel{\mathfrak{S}}\lar R^{m-n}\to \coker \mathfrak{S}\to 0,
	\end{eqnarray}
	where
	$$F_i=({\rm Sym}_{i-1}(R^{m-n})\otimes \wedge^{m-n+i} R^{m-1}) (-(m-n+i))\simeq R(-(m-n+i))^{\beta_i},\quad (1\leq i\leq n-1) $$
	and ${\rm Sym}_i(-)$ denote the $i$th symmetric power.

	Since  $\left[\begin{matrix}
		f_{n+1}&\cdots&f_m
	\end{matrix}\right]\; \mathfrak{S}=-\left[\begin{matrix}\frac{\partial F}{\partial x_1}&\cdots&\frac{\partial F}{\partial x_n}\end{matrix}\right]\,\mathfrak{S}'$ we deduce from \eqref{B-R-complex} 
	the following complex
	
	$$	0\to F_{n-1}(-(m-1)) \stackrel{\partial}\lar \cdots \stackrel{\partial}\lar F_{1}(-(m-1))\stackrel{\mathfrak{S}'\eta}\lar R(-(m-1))^{n}\to R\to R/J_F\to 0.$$
	
	Granted the assertions in the two claims, it follows from \cite[Theorem D]{AndSim1986} that the this complex is a minimal graded free resolution of  $R/J_F.$  	Therefore,  $R/J_F$ has homological dimension $n$, as claimed, and in addition it has a minimal graded free resolution of the stated form.
\end{proof}

We also recover one of the results of \cite{BuSiTo2022}:

\begin{Theorem} \label{minimal_reduction}	{\rm (char$(\K)\nmid m$)}\;
	Let $F\in \K[x_1,\ldots,x_n]$ stand for the defining form of a generic hyperplane arrangement of size $m\geq n+1$.
	Then the gradient ideal $J_F$ is a  minimal reduction of the ideal $I$ of $(m-1)$-fold products of the defining  forms of $F$.
\end{Theorem}
\begin{proof} We first prove the following
	

	As before, we can write $$\left[\begin{matrix} \mathbb{I}_{n}&\boldsymbol0\\-A&\mathbb{I}_{m-n}\end{matrix}\right]\Gamma=\left[\begin{matrix}\mathfrak{S}'\\
		\mathfrak{S}\end{matrix}\right]$$ 
	where  $\mathfrak{S}'$ is a $n\times (m-1)$ matrix of linear forms.
	
Since
	$$\left[\begin{matrix}\frac{\partial F}{\partial x_1}&\cdots&\frac{\partial F}{\partial x_n}&f_{n+1}&\cdots f_m\end{matrix}\right]\left[\begin{matrix}\mathfrak{S}'\\
		\mathfrak{S}\end{matrix}\right]=\boldsymbol{0},$$ 
	then
	\begin{equation*}
		\left[\begin{matrix}
			f_{n+1}&\cdots&f_m
		\end{matrix}\right]\; \mathfrak{S}
		= -\left[\begin{matrix}\frac{\partial F}{\partial x_1}&\cdots&\frac{\partial F}{\partial x_n}\end{matrix}\right]\mathfrak{S}'.
	\end{equation*}
	It follows that, for every $(m-n)\times (m-n)$ submatrix $B$ of $\mathfrak{S}$, one has
	$$I_1(\left[\begin{matrix}
		f_{n+1}&\cdots&f_m
	\end{matrix}\right]B\,{\rm adj}(B))\subset \langle x_1,\ldots, x_n\rangle^{m-n}J_F.$$
	Thus,
	$$(\det B)\langle f_{n+1},\ldots,f_m\rangle\subset \langle x_1,\ldots, x_n\rangle^{m-n}J_F.$$
	Since $B$ is arbitrary, then Proposition~\ref{minors_are_monomials} entails the inclusion
	\begin{equation}
		\langle x_1,\ldots, x_n\rangle^{m-n}\langle f_{n+1},\ldots,f_m\rangle\subset \langle x_1,\ldots, x_n\rangle^{m-n} J_F.
	\end{equation}
We translate this inclusion into an expression of integral dependence.	
Namely, let $M_1,\ldots,M_N$ denote the set of monomials of $R$  of degree $m-n.$
	Then
	$$M_jf_i=L_{j,1} M_1+\cdots +L_{j,N} M_N, \quad 1\leq j\leq N, \, n+1\leq i\leq m,$$
	for certain $\K$-linear forms $L_{j,l}\in R$ in  the partial derivatives $\frac{\partial F}{\partial x_1},\ldots,\frac{\partial F}{\partial x_n}$.
	
	These relations can be read as
	$$\left(f_i\mathbb{I}_N- (L_{j,l})_{1\leq j,l\leq N}\right)\left[\begin{matrix}M_1&\cdots&M_N\end{matrix}\right]^t=\boldsymbol0,$$
	or yet,
	$$\det\left(f_i\mathbb{I}_N-(L_{j,l})_{1\leq j,l\leq N}\right)\langle x_1,\ldots,x_n\rangle^{m-n}=0.$$
	Hence,
	$$\det\left(f_i\mathbb{I}_N-(L_{j,l})_{1\leq j,l\leq N}\right)=0$$
	which is an equation of integral dependence of $f_i$ over $J_F$.  In particular, $J_F$ is a reduction of $I.$
	The minimality follows from the fact that, since the arrangement is generic, then the ideal $I$ has maximal analytic spread (see \cite{Ter2002}, also \cite{ScTo2009}).
\end{proof}

\section{General forms}\label{Section3}

In this section we analyze the vanishing of the depth of the gradient ideal of the product of general  forms.

\subsection{Preliminaries}
For the reader's convenience we briefly recall the notion of parameter spaces of polynomial forms.
Let $d\geq 1$ be an integer. For $R$ as above, a form $f\in R$ of degree $d$ depends on ${d+n-1 \choose n-1}$ coefficients.
Thinking of these coefficients as ${d+n-1 \choose n-1}$  new variables $y_{i_1,\ldots,i_n}$ over $\K$, consider the polynomial ring $\mathbb{S}$ over $\K$ in these indeterminates.
Then ${\rm Proj}_{\K}(\mathbb{S})=\Proj_{\K}^N$,
where $N={d+n-1 \choose n-1}-1$, is called the {\em parameter space} of forms of degree $d\geq 1$ in $n$ variables.
Note the tautological map $\tau_d: R_d\lar \Proj_{\K}^N$ which associates to a nonzero form $f\in R$ of degree $d\geq 1$ the point of $\Proj_{\K}^N$ whose coordinates correspond to the  coefficients of $f$ in $\K$. Since we don't distinguish between a nonzero form $f$ of degree $d\geq 1$ and $\alpha f$, for any nonzero scalar $\alpha\in \K$, this map is well-defined.

Drawing upon these matters,  one says that the {\em general} form of degree $d\geq 1$ in $R$ has a certain property  if the set of points $\tau_d(f)$ such that $f$ has this property contains a dense open set of $\Proj_{\K}^N$ {\rm (}i.e.,  avoids a proper closed set of $\Proj_{\K}^N${\rm )}.

These concepts extend to a finite set of nonzero forms $\{f_1,\ldots,f_m\}\subset R$ of varying degrees $d_i\geq 1$ ($1\leq i\leq m$). Namely, letting $\Proj_{\K}^{N_i}$ denote the parameter space for $R_{d_i}$, one considers the product space
\begin{equation}\label{parameter_space} \mathbb{P}_{\K,\boldsymbol{d}}:=\Proj_{\K}^{N_1}\times_{\K}\cdots \times_{\K} \Proj_{\K}^{N_m}, \: N_i={d_i+n-1 \choose n-1}-1,
		\end{equation}
along with the  tautological map
\begin{equation} \label{taut_map}\tau_{\bf d}:=\tau_{d_1}\times \cdots \times \tau_{d_m}: R_{d_1}\times_{\K}\cdots \times_{\K} R_{d_m}\lar\mathbb{P}_{\K,\boldsymbol{d}},
	\end{equation}
where $\tau_{d_i}: R_{d}\lar \Proj_{\K}^{N_i}$  is as above.

One similarly talks about the set of $m$ general forms of degrees $d_i\geq 1$ ($1\leq i\leq m$) in $R$ having a certain property  in terms of avoidance of a suitable proper closed set defined by multi-homogeneous forms.

With the usual abuse of calling a form in $R$ {\em smooth} if $V(f)\subset {\rm Proj}_{\K} (R)=\Proj_{\K}^{n-1}$ is non-singular, one has the well-known fact that the general form of degree $d\geq 1$ in $R$  is smooth; a reference for this result in zero characteristic is \cite[Exercise 17.17 (a)]{HarrisBook}.
It is also  established by means of the discriminant of singular hypersurfaces (see, e.g., \cite[Chapter 7, Section 7.1]{EisHar3264}).)

Clearly, then $V(f)$ is a normal hypersurface provided $n\geq 4$, in particular $f$ is an irreducible polynomial thereof.
One naturally asks about the case of several forms, say, $\mathbf{f}:=\{f_1,\ldots,f_m\}\subset R=\K[x_1,\ldots,x_n]$ such that the projective variety $V(\mathbf{f})\subset \Proj_{\K}^{n-1}={\rm Proj}_{\K}(R)$ is smooth.
For $n\geq 4$,  the ring $R/\langle \mathbf{f}\rangle$ is in particular a normal ring (domain, in fact), the Serre normality criterion forces $m\leq n-2$.
Therefore, if we wish that the  set $\mathbf{f}$ of $m$ general forms be such that $V(\mathbf{f})$ is smooth then at most $m\leq n-2$.
In this case, in fact this is a theorem (see \cite[Exercise 17.17 (b)]{HarrisBook} in zero characteristic).
This theorem says in addition that the forms intersect transversely everywhere, in particular, each $f_i$ is smooth and $\{f_1,\ldots,f_m\}$ is a regular sequence (\cite[Theorem 17.18]{HarrisBook}).

The following result might be part of the general forms `folklore', but we give a short argument for the reader's convenience.

\begin{Proposition}\label{general_is_regular_sequence}
The set ${\bf f}=\{f_1,\ldots,f_m\}\subset R=\mathbb{K}[x_1,\ldots,x_n]$ of general forms of arbitrary degrees with $m\leq n$  is a regular sequence.
\end{Proposition}  
\begin{proof} The minimal graded free resolution of a complete intersection in $R$ is given by the Koszul complex. Therefore, the Hilbert series of the ideal $\langle{\bf f}\rangle$ depends only on $m$ and the degrees $d_1,\ldots,d_m$ of the forms (see, e.g., \cite[Proposition 7.4.11]{SimisCA2023}).
Thus, writing  $\sum_{j\geq0 }a_jt^j$ for this Hilbert series, one has that  ${\bf f}$ is a regular sequence if, and only if, $HS_{R/\langle {\bf f}\rangle}(t)=\sum_{j\geq0 }a_jt^j.$  On the other hand, since the Hilbert function is polynomial, it follows that ${\bf f}$ is a regular sequence if and only if for $N\gg 0$, one has $\dim_{\mathbb K} [R/\langle{\bf f}\rangle]_{j}=a_j$, or equivalently,
$$\dim_{\mathbb K} (f_1R_{j-d_1}+\cdots +f_{m}R_{j-d_m})=\dim R_j-a_j,$$
for every $0\leq j\leq N.$

Now, let $B_j$ denote the coefficient matrix of the inclusion $f_1R_{j-d_1}+\cdots +f_{m}R_{j-d_m}\subset R_j.$  Then, $\dim_{\mathbb K} (f_1R_{j-d_1}+\cdots +f_{m}R_{j-d_m})=\dim R_j-a_j$ if and only if $\rk B_j=\dim R_j-a_j.$
This is an open condition in the parameter space $\mathbb{P}^{d_1+n-1\choose n-1}\times \cdots\times \mathbb{P}^{d_m+n-1\choose n-1}$ of the coefficients of the forms.
In order to show that this condition is satisfied by a {\em nonempty} Zariski open set it suffices to exhibit one single ``point''  of such a set of forms, e.g., $x_1^{d_1},\ldots,x_{m}^{d_m}$.
\end{proof}

\subsection{The Jacobian ideal at large}\label{Jacobian_general}

Set $R=\mathbb{K}[x_1,\ldots,x_n]$ as previously and keep the definition and notation of (\ref{parameter_space} and (\ref{taut_map})).

To proceed, introduce the following additional notation.


$\bullet$  
For any set of integers $2\leq d_1\leq\cdots\leq d_m$,  consider the multi-polynomial ring
$$\mathbb{S}:=\K[y_{\alpha_1,\ldots,\alpha_n}^{(d_1)}\,|\,\alpha_1+\cdots+\alpha_n=d_1]\otimes \cdots\otimes \K[y_{\alpha_1,\ldots,\alpha_n}^{(d_m)}\,|\,\alpha_1+\cdots+\alpha_n=d_m],$$
and write $$\mathfrak{f}_j:=\sum_{\alpha\in\NN^n\atop |\alpha|=d_j} y_{\alpha_1,\ldots,\alpha_n}^{(d_j)}x_1^{\alpha_1}\cdots x_{n}^{\alpha_n}$$ for the generic  form of degree $d_j$, a polynomial in the ring $\mathbb{S}[x_1,\ldots,x_n]$.

$\bullet$ $\Theta(\boldsymbol{\mathfrak{f}})$ is the Jacobian matrix of the set $\boldsymbol{\mathfrak{f}}=\{\mathfrak{f}_1,\ldots,\mathfrak{f}_m\}$ with respect to the variables $x_1,\ldots,x_n.$

$\bullet$  $\boldsymbol{\mathfrak{g}}$ is the set of all $m$-minors of $\Theta(\boldsymbol{\mathfrak{f}})$ (in particular, $\boldsymbol{\mathfrak{g}}=\emptyset$ if $m> n$). Note that the coefficients of the polynomials in  $\boldsymbol{\mathfrak{g}}$ are multi-graded  polynomials of $\mathbb{S}.$

$\bullet$ $\delta:=\max\{\deg_{\xx}\mathfrak{h}\,|\,\mathfrak{h}\in \boldsymbol{\mathfrak{g}}\cup \boldsymbol{\mathfrak{f}} \}.$

$\bullet$ Bring out the ideal $\langle I_m(\Theta(\boldsymbol{\mathfrak{f}})), 
\boldsymbol{\mathfrak{f}}\rangle\subset \mathbb{S}[x_1,\ldots,x_n]$, considered as an $\mathbb{S}$-module.

For $t\geq \delta,$ a ``natural'' generating set (not necessarily minimal) of this $\mathbb{S}$-module  in degree $t$ is the set $$\{\mathfrak{h}\,x_1^{\alpha_1}\cdots x_{n}^{\alpha_n}\,|\,\mathfrak{h}\in \boldsymbol{\mathfrak{g}}\cup \boldsymbol{\mathfrak{f}},\,  \alpha_1+\cdots+\alpha_n=t-\deg\mathfrak{h}\},$$ 
whose cardinality (eventually,  counting repetitions) is 
$$N(t,{\boldsymbol{d}})=\sum_{\mathfrak{h}\in \boldsymbol{\mathfrak{g}}\cup \boldsymbol{\mathfrak{f}}}{t-\deg\mathfrak{h}+n-1\choose n-1}.$$

$\bullet$ Let $\mathfrak{N}_{ t,\boldsymbol{d}}$ denote the $N(\mathfrak{i},{\boldsymbol{d}})\times {t+n-1\choose n-1}$ matrix whose columns are indexed by the natural monômios generators of the $\K$-vector space $\K[x_1,\ldots, x_n]_t$, and whose rows are indexed by the set of standard monomial generators of $\langle I_m(\Theta(\boldsymbol{\mathfrak{f}})), 
\boldsymbol{\mathfrak{f}}\rangle$ as given above.

This way, $\mathfrak{N}_{ t,\boldsymbol{d}}$ is like a content matrix with respect to the parameter variables, formed by two vertical blocks: in the upper one any row has null entries except one, which is $1$, while the  lower block is a linear sparse matrix in the parameter variables.

$\bullet$ Finally, for our purpose the crucial value of the degree $t$ will be $t_0:=n(\delta-1)+1$, which is the socle degree plus one of a regular sequence of $n$ forms of degree $\delta$.
For this value of $t$ one has $ {t+n-1\choose n-1}={n\delta\choose n-1}:=r$.

In the next theorem we deal with forms in $R=\K[x_1,\ldots,x_n]$, keeping the same notation as above for the Jacobian matrix with respect to $x_1,\ldots,x_n$, hoping no confusion arises.

\begin{Theorem}\label{basic_open_set}
	With the above notation, set $U_{t_0}:= \Proj_{\K,\mathbf{d}} \setminus V(I_r(\mathfrak{N}_{ t_0,\boldsymbol{d}})).$
	Then$:$
	\begin{enumerate}
		\item[\rm (a)] Given a set of forms $\mathbf{f}=\{f_1,\ldots,f_m\}\subset R $  of degrees $2\leq d_1\leq\cdots\leq d_m$, then the ideal $\langle I_m(\Theta(\mathbf{f})), 
		\mathbf{f}\rangle\subset R$ has height $n$ if and only if $\tau_{\bf d}(f_1,\ldots, f_m)\in U_{t_0}$.
		\item[\rm (b)] 
		The Zariski open subset $U_{t_0}$ is nonempty.
	\end{enumerate}
\end{Theorem} 
\begin{proof} (a) If the ideal
	If $\langle I_m(\Theta(\mathbf{f})), 
	\mathbf{f}\rangle$ has height $n$, it contains a regular sequence $g_1,\ldots,g_n$ of forms with $\deg g_i=\delta$ for every $1\leq i\leq m.$ Since the socle degree of  $R/\langle g_1,\ldots,g_n\rangle$ is $n(\delta-1)$ it follows that 
	$$R_{t_0}=[\langle g_1,\ldots,g_n\rangle]_{t_0}\subset [\langle I_m(\Theta(\mathbf{f})), 
	\mathbf{f}\rangle]_{t_0}\subset R_{t_0},$$
	hence, $[\langle I_m(\Theta(\mathbf{f})), 
	\mathbf{f}\rangle]_{t_0}= R_{t_0}.$
	From this follows that the matrix obtained by evaluating the parameter entries of  $\mathfrak{N}_{ t_0,\boldsymbol{d}}$ at the corresponding coefficients of the set $\{f_1,\ldots, f_m\}$ has rank $r.$ Thus,  $\tau_{\bf d}(f_1,\ldots, f_m)\in U_{t_0}.$
	
	Conversely, if $\tau_{\bf d}(f_1,\ldots, f_m)\in U_{t_0}$ then the the matrix obtained by evaluating the parameter entries of  $\mathfrak{N}_{ t_0,\boldsymbol{d}}$ at $\tau_{\bf d}(f_1,\ldots, f_m)$  has rank $r.$ Hence, the inclusion $[\langle I_m(\Theta(\mathbf{f})), 
	\mathbf{f}\rangle]_{t_0}\subset  R_{t_0}$ is an equality, thus implying that the ideal $\langle I_m(\Theta({\bf f})), f_1,\ldots,f_m\rangle$ has a maximal regular sequence in degree $t_0$.
	
	(b) In other words, by (a), it suffices to show the existence of a set $\mathbf{f}$ of $m$ forms in $R$ such that the ideal $\langle I_m(\Theta(\mathbf{f}), \mathbf{f}\rangle$ has height $n$.
	
	We induct on $n.$ 
	 If $n=2$ then for $m=1$ we can take $f_1=x_1^2+x_2^2$, and for $m\geq 2$  we can, by Proposition \ref{general_is_regular_sequence}, take $\{f_1,\ldots,f_m\}$ such that $\{f_1,f_2\}$ is a regular sequence. 
	 
	Thus, assume that $n>2.$ By the inductive step, there is a set of forms $\mathbf{f}=\{f_1,\ldots,f_m\}\subset \K[x_1,\ldots,x_{n-1}] $ of degrees $2\leq d_1\leq \cdots \leq d_m$ such that 
	$$\Ht  \langle I_m(\Theta(\mathbf{f})), 
\mathbf{f}\rangle=\Ht \langle I_{m-1}(\Theta({\mathbf{f}}\setminus{f_m})), f_1,\ldots,f_{m-1}\rangle=n-1 $$
	Now, consider the deformed set of forms $\widetilde{\mathbf{f}}=\{f_1,\ldots,f_{m-1},f_m+x_n^{d_m}\}$ of $\K[x_1,\ldots,x_n].$ 
	We claim that $\Ht \langle I_m(\Theta(\widetilde{\bf f}), \widetilde{\bf f} \rangle=n$.
	
	Indeed, the Jacobian matrix $\Theta(\widetilde{\mathbf{f}})$ has the shape:
	$$
	\left[\begin{array}{ccc}
		\Theta({\mathbf{f}}\setminus{f_m})&\boldsymbol0\\
		\Theta(f_m)&d_mx^{d_m-1}	
	\end{array}\right].$$
	Clearly, then
	\begin{equation}\label{estimativa}
		\langle I_m(\Theta({\bf f})),x^{d_m-1}I_{m-1}(\Theta({\bf f}\setminus{f_m})),f_1,\ldots,f_m+x_n^{d_m}\rangle \subset \langle I_m(\Theta(\widetilde{\bf f}), \widetilde{\bf f} \rangle. 
	\end{equation}
	Now, let $P\subset R$ be a prime ideal of $R$ containing  $\langle I_m(\Theta(\widetilde{\bf f}), \widetilde{\bf f} \rangle$.  In particular,  by \eqref{estimativa}, it contains  $x^{d_m-1}I_{m-1}(\Theta({\bf f}\setminus{f_m})).$ 
	 If $P$ contains $x_n$ then, again by \eqref{estimativa}, it contains the ideal  $\langle I_m(\Theta({\bf f})),f_1,\ldots,f_m, x_n\rangle\subset P$ of height $n=n-1+1$ by the inductive assumption and since $x_n$ is a regular element over $R/\langle I_m(\Theta({\bf f})),f_1,\ldots,f_m\rangle$.

	On the other hand, if $I_{m-1}(\Theta({\bf f}\setminus{f_m}))\subset P$ then, always by \eqref{estimativa}, $$\langle I_{m-1}(\Theta({\bf f}\setminus{f_m})),f_1,\ldots,f_{m-1},f_m+x_n^{d_m}\rangle\subset P,$$
	 and again, by the inductive hypothesis, one has $\Ht P=n.$
	\end{proof}

\begin{Proposition}\label{codim_jac}
	The  set ${\bf f}=\{f_1,\ldots,f_m\}\subset R$ of general forms  satifies the equality $\Ht \langle I_{m}(\Theta({\bf f})),{\bf f}\rangle =n$ {\rm (}in particular, the variety $V(\bf f)\subset \pp_{\K}^{n-1}$ is a reduced and irreducible smooth variety provided $m\leq n-2${\rm )}.
\end{Proposition}
\begin{proof}
	This is an immediate consequence of Theorem~\ref{basic_open_set}. 
\end{proof}

\subsection{Main result}

We next move on to the RTY-property, which is our main goal.

As above, we will make constant use of the Jacobian matrix $\Theta(\bf f)$ of a finite set $\bf f$ of forms of $R$.

Let there be given a set  ${\bf f}=\{f_1,\ldots,f_m\}\subset R$  of forms, with respective degrees $d_1,\ldots,d_m$ ($d_i\geq 2$ for every $i$).
We observe that in the sequel we do not assume $m\leq n$.

Given a permutation $\sigma$ of $m$ indices, introduce the following additional notation
\begin{enumerate}
	\item[$\bullet$] ${\bf f}_{\sigma} : = \{f_{\sigma(1)},\ldots, f_{\sigma(m)}\}.$
	\item[$\bullet$] $\Pi({\bf f}_{\sigma}):=\{ f_{\sigma(2)}\cdots f_{\sigma(m)},\ldots, f_{\sigma(1)}\cdots f_{\sigma(m-1)}\},$ that is,  the orderly set of $(m-1)$-fold products of ${\bf f}_{\sigma}.$ 
	\item[$\bullet$] $\mathfrak{D}({\bf f}_{\sigma}):$  the $m\times(m-1)$ matrix 
	$$
	\left[\begin{array}{cccccc}
		-f_{\sigma(1)}&0&\cdots&0\\
		0&-f_{\sigma(2)}&\cdots&0\\
		\vdots&\vdots&\ddots&\vdots\\
		0&0&\cdots&-f_{\sigma(m-1)}&\\
		f_{\sigma(m)}&f_{\sigma(m)}&\cdots&f_{\sigma(m)}
	\end{array}\right],$$
	if $m\geq 2$, and $\mathfrak{D}({\bf f}_{\sigma})=\emptyset$ if $m=1.$
	\item[$\bullet$] $\widetilde{\mathfrak{D}({\bf f}_{\sigma})}:$   the $(m-2)\times (m-2)$ upper left corner  (diagonal) submatrix of $\mathfrak{D}({\bf f}_{\sigma})$
	\item[$\bullet$] ${\rm M}({\bf f}_{\sigma}):$  the $m\times (n+m-1)$ block matrix $\left[\begin{array}{c|c}\Theta({\bf f}_{\sigma})&\mathfrak{D}({\bf f}_{\sigma})\end{array}\right].$
\end{enumerate}	
If $\sigma$ is the identity we drop the subscript throughout the abobe notation. 

	In addition, set 
$\mathbb{I}_{r}$ for the $r\times r$ identity matrix.

Note that the elements of $\Pi({\bf f}_{\sigma})$ are the signed ordered $(m-1)$-minors of the $m\times(m-1)$ matrix $\mathfrak{D}({\bf f}_{\sigma}).$ Thus, by the Hilbert-Burch theorem, if $\Ht \langle \Pi({\bf f}_{\sigma}) \rangle\geq 2$ then $\mathfrak{D}({\bf f}_{\sigma})$ is the syzygy matrix of the ideal $\langle \Pi({\bf f}_{\sigma}) \rangle.$ On the other hand, since the set $\Pi({\bf f}_{\sigma})$ is a permutation of the set $\Pi({\bf f}),$ we have $\langle\Pi({\bf f}_{\sigma})\rangle=\langle\Pi({\bf f})\rangle$  and
\begin{equation}\label{conjugatio1}
	\mathfrak{D}({\bf f}_{\sigma})=A\, \mathfrak{D}({\bf f}) B
\end{equation}
where $A\in{\rm GL}_{m}(R)$ and $B\in {\rm GL}_{m-1}(R)$ correspond to suitable elementary row and  column  operations, respectively.  In particular,

\begin{equation}\label{conjugation2}
	M({\bf f}_{\sigma})=A\, M({\bf f}) \left[\begin{array}{cc} \mathbb{I}_{n}&\\&B\end{array}\right].
\end{equation}

\begin{Lemma}\label{conjugation} Suppose that any two among the forms $f_1,\ldots,f_m$ are relatively prime. Then, for every $1\leq i\leq m,$ the matrix $M({\bf f})$ is conjugate to the block matrix
	$$
	\left[\begin{array}{ccc}\Theta({\bf f}\setminus \{f_i\})&\mathfrak{D}({\bf f}\setminus\{f_i\})& \boldsymbol0\\
		\Theta(f_i)	&\boldsymbol0&f_{i}\\
	\end{array}\right],$$
	where $\boldsymbol0$ denotes a null matrix of the suitable dimension. 
\end{Lemma}
\begin{proof} For any permutation $\sigma$ we have a block decomposition
	
	$$M({\bf f}_{\sigma})=\left[\begin{array}{ccccc}
		\Theta({\bf f}_{\sigma}\setminus \{f_{\sigma(m-1)},f_{\sigma(m)}\})&\widetilde{\mathfrak{D}({\bf f}_{\sigma})}&\boldsymbol0\\
		\Theta(f_{\sigma(m-1)})&\boldsymbol0&-f_{\sigma(m-1)}\\
		\Theta(f_{\sigma(m)})&f_{\sigma(m)}\boldsymbol{1}_{m-2}&f_{\sigma(m)}
	\end{array}\right].$$
	Selecting the matrix
	
	$$U=\left[\begin{array}{cccccc}
		\mathbb{I}_{n}&\boldsymbol0&\boldsymbol0\\
		\boldsymbol0&\mathbb{I}_{m-2}&\boldsymbol0\\
		\boldsymbol0&\boldsymbol{1}_{m-2}&1
	\end{array}\right]\in {\rm GL}_{n+m-1}(R),$$
clearly,
	
	$$M({\bf f}_{\sigma})U=
	\left[\begin{array}{ccccccccccc}
		\Theta({\bf f}_{\sigma}\setminus \{f_{\sigma(m-1)},f_{\sigma(m)}\})&\widetilde{\mathfrak{D}({\bf f}_{\sigma})}&\boldsymbol0\\
		\Theta(f_{\sigma(m-1)})&f_{\sigma(m-1)}\boldsymbol{1}_{m-2}&-f_{\sigma(m-1)}\\
		\Theta(f_{\sigma(m)})&\boldsymbol0&f_{\sigma(m)}
	\end{array}\right].$$
Next, letting
$${\bf D}^t=\left[\begin{array}{cccccccccc}d_{\sigma(1)}&\cdots&d_{\sigma(m-2)}\end{array}\right]\quad\mbox{and}\quad \delta=\sum_{j=1}^{m-1}d_{\sigma(j)},$$
consider the matrix	
	$$V=\left[\begin{array}{ccccc}
		\mathbb{I}_{n}&\boldsymbol0&{\bf x}^t\\
		\boldsymbol0&\mathbb{I}_{m-2}&{\bf D}\\
		\boldsymbol0&\boldsymbol0&\delta
	\end{array}\right]\in {\rm GL}_{n+m-1}(R),$$
where ${\bf x}= \left[x_1 \; \cdots \; x_n\right]$.
	
Euler's formula implies the following relations: 
	\begin{equation*}
		\Theta({\bf f}_{\sigma}\setminus \{f_{\sigma(m-1)},f_{\sigma(m)}\})\cdot{\bf x}^t+\widetilde{\mathfrak{D}({\bf f}_{\sigma})}{\bf D}=\boldsymbol0,
	\end{equation*}
	\begin{equation*}
		\Theta(f_{\sigma(m-1)})\cdot {\bf x}^t+ f_{\sigma(m-1)}\boldsymbol1_{m-2}\cdot {\bf D}-\delta f_{\sigma(m-1)}=0
	\end{equation*}
	and
	\begin{equation*}
		\Theta(f_{\sigma(m)})\cdot{\bf x}^t+\delta f_{\sigma(m)}=(\sum_{j=1}^m d_{\sigma(j)})f_{\sigma(m)}.
	\end{equation*}
	From these,
	\begin{equation*}\label{conjugation3}
		{\rm M}({\bf f}_{\sigma})UV=\left[\begin{array}{ccccccccccc}
			\Theta({\bf f}\setminus \{f_{\sigma(m-1)},f_{\sigma(m)}\})&\widetilde{\mathfrak{D}({\bf f}_{\sigma})}&\boldsymbol0\\
			\Theta(f_{\sigma(m-1)})&f_{\sigma(m-1)}\boldsymbol{1}_{m-2}&\boldsymbol0\\
			\Theta(f_{\sigma(m)})&\boldsymbol0&(\sum_{j=1}^m d_{\sigma(j)})f_{\sigma(m)}
		\end{array}\right]
	\end{equation*}
	Applying with $\sigma:=\sigma_i$ such that  
	${\bf f}_{\sigma_i}= \{f_1,\ldots,f_{i-1},f_{i+1},\ldots,f_{m},f_i\},$ gives:

	
	\begin{eqnarray}\label{almost_final}\nonumber
		{\rm M}({\bf f}_{\sigma})UV&=&
		\left[\begin{array}{ccccccccccc}
		\Theta({\bf f}\setminus \{f_{m},f_{i}\})&\widetilde{\mathfrak{D}({\bf f}_{\sigma_i})}&\boldsymbol0\\
		\Theta(f_{m})&[f_m\cdots f_m]&\boldsymbol0\\
		\Theta(f_{i})&\boldsymbol0&(\sum_{j=1}^m d_{\sigma(j)})f_{i}
	\end{array}\right]\\ [5pt]
	&=& \left[\begin{array}{ccccccccccc}
		\Theta({\bf f}\setminus \{f_{m},f_{i}\})&\mathfrak{D}({\bf f}\setminus\{f_i\})&\boldsymbol0\\
		\Theta(f_{i})&\boldsymbol0&(\sum_{j=1}^m d_{\sigma(j)})f_{i}
	\end{array}\right]
	\end{eqnarray}
By \eqref{conjugation2}, $M({\bf f})$ is then conjugate to (\ref{almost_final}).
Since $d:=\sum_{j=1}^m d_{\sigma(j)}\neq 0$, multiplying on the right by the matrix
$$\left[\begin{array}{cc} \mathbb{I}_{n+m-2}&\\&1/d\end{array}\right]$$
yields the required matrix in the statement. 
\end{proof}

\begin{Proposition}\label{heights_main}
Let ${\bf f}=\{f_1,\ldots,f_m\}\subset R:=\mathbb K[x_1,\ldots,x_n]$ be a set  of $m\geq 1$  forms of degrees $2\leq  \deg f_1=d_1\leq \cdots\leq \deg f_m=d_m$.
Suppose that  any of its subsets with at most $n$ elements is a regular sequence satisfying the equality of {\rm Propositon~\ref{codim_jac}}. Then, $\Ht I_{m-1}(\mathfrak{D}({\bf f}))=2$ and $\Ht I_{m}(M({\bf f}))=n.$
\end{Proposition}
\begin{proof} Since any  subset of ${\bf f}$ with at most $n$ elements is a regular sequence satisfying the equality of {\rm Propositon~\ref{codim_jac}} then $F=f_1\cdots f_m$ is a reduced homogeneous polynomial. Hence, $\Ht J_F\geq 2.$ In particular, since $J_F\subset I_{m-1}(\mathfrak{D}({\bf f}))$ we have   $\Ht I_{m-1}(\mathfrak{D}({\bf f}))\geq 2$ as well.

Let us argue that $I_{m}(M({\bf f}))$ has height $n.$ 
	
	We induct on $m$. For $m=1$, one has $ I_1(M({\bf f}))=\langle f_1,I_1(\Theta(f_1))$, which, by hypothesis,  is an ideal of height $n.$

Now, suppose $m>1.$ By  Lemma~\ref{conjugation}, one has
\begin{equation}\label{inclusion}
	f_iI_{m-1}(M({\bf f}\setminus\{f_i\}))\subset I_{m}(M({\bf f})).
\end{equation}
for every $1\leq i\leq m.$ Thus, if  $P\subset R$ is a minimal prime of $I_{m}(M({\bf f}))$ then:
\begin{enumerate}
	\item[{\rm (i)}] Either $I_{m-1}(M({\bf f}\setminus\{f_i\}))\subset P$ for some $1\leq i\leq m,$
	\item[{\rm (ii)}] else, $f_{i}\in P$ for every $1\leq i\leq m.$
\end{enumerate}

If (i) is the case, then:
$$	\Ht P\geq\Ht I_{m-1}(M({\bf f}\setminus\{f_i\}))
=n,$$
where the last equality is by induction on $m$.  Hence, $\Ht P=n.$

Now, suppose (ii) is satisfied. If $m\leq n$, then $\langle I_{m}(\Theta({\bf f})),{\bf f}\rangle\subset P$; hence, by the standing assumption, $\Ht P=n$. On the other hand, if $m>n$ then $f_1,\ldots,f_n$ is a regular sequence again by the main assumption. Hence, $\Ht P=n$ too.

Therefore, $\Ht I_{m}(M({\bf f}))=n$
\end{proof}

We now derive the main theorem of the section.

\begin{Theorem}\label{main_general_forms}
Let $\{f_1,\ldots,f_m\}\subset R:=\mathbb K[x_1,\ldots,x_n]$ be a set of forms of degrees $\geq 2$, such that any subset with at most $n$ elements is  general, and let   $F:=f_1\cdots f_m$. Then, a minimal  free resolution of $R/J_F$ has the form 
$$0\to R^{b_n}\stackrel{\partial_n}{\to} \cdots\to R^{b_2}\stackrel{\partial_2}{\to} R^{n}\stackrel{\partial_1}{\to} R\lar R/J_F\to 0$$
where $b_i={m+n-1\choose m+i-1}{m+i-3\choose m-1},\, 2\leq i\leq n.$ In particular, $\{f_1,\ldots,f_m\}$ satisfies the {\rm RTY}-property.
\end{Theorem}
\begin{proof}
The gradient ideal $J_F$ of the product $F=f_1\cdots f_m\in R$ is generated by the $m$-minors
 $$\det \left[\begin{array}{cccccc}
 \partial f_1/\partial x_i &	-f_{1}&0&\cdots&0\\
  \partial f_2/\partial x_i &	0&-f_{2}&\cdots&0\\
\vdots & 	\vdots&\vdots&\ddots&\vdots\\
 \partial f_{m-1}/\partial x_i & 	0&0&\cdots&-f_{m-1}\\
 \partial f_m/\partial x_i & 	f_{m}&f_{m}&\cdots&f_{m}
 \end{array}\right],
 $$
 for $1\leq i\leq m$.
 In other words, $J_F$ is the ideal generated by the $m$-minors of $M(\bf f)$ fixing the last $m-1$ columns. 
 
 On the other hand,  by Proposition~\ref{general_is_regular_sequence} and Proposition~\ref{codim_jac},, the required equalities in Proposition~\ref{heights_main} are satisfied. Since $n= (n+m-1)-m +1$, we are fit to apply  the main theorem of \cite{AndSim1981}, giving a free resolution of $R/J_F$ as asserted.
 
 It remains to argue that such a resolution is minimal.
 For this, we note that this resolution is obtained from the Buchsbaum--Rim resolution
 	{\scriptsize$$0\to \bigwedge^{m+n-1} R^{m+n-1}\otimes {\rm Sym}_{n-2}((R^m)^{\ast})\stackrel{\partial_n'}\lar\cdots\stackrel{\partial_3'}\lar \bigwedge^{m+1}R^{m+n-1}\otimes {\rm Sym}_0((R^m)^\ast)\stackrel{\partial_2'}\lar R^{m+n-1}\stackrel{M({\bf f})}\lar R^m, $$} 
such that:

$\bullet$ The entries of $\partial_i'$ ($3\leq i\leq n$) are $\mathbb{Z}$-linear combinations of the entries of $M({\bf f});$

$\bullet$ $\partial_2$ is the composite $\wedge^{m+1}R^{m+n-1} \stackrel{\partial_2'}\lar R^{m+n-1} \stackrel{p}\lar R^n$ where $p$ is the natural projection.

$\bullet$ $\partial_2'(e_{{\bf i}})=\sum_{j=1}^{m+1} \pm1 \det M({\bf f})_{{\bf i}\setminus i_{j}}e_{i_j}$, 
where  $M({\bf f})_{{\bf i}\setminus i_{j}}$ is the $m\times m$ matrix of $M({\bf f})$ indexed by the columns of the set ${\bf i}\setminus i_{j}$.

From these properties, it is clear that the resolution is minimal, and, in fact, graded.
\end{proof}

\begin{Example}\rm Let $n=3.$  In this case, fixing the  following bases 
	$$\{e_1,\ldots,e_{m+2}\} \quad\mbox{and}\quad\{e_2\wedge \cdots \wedge e_{m+2},\cdots, e_1\wedge \cdots\wedge \widehat{e}_{i}\cdots \wedge e_{m+2}, e_1\wedge \cdots \wedge e_{m+1} \}$$
	for  $R^{m+2}$ and $\bigwedge^{m+1}R^{m+2},$ respectively,  the  matrix representing  $\partial'_2$ is an $(m+2)\times(m+2)$ anti-symmetric matrix. In particular, the Buchsbaum-Rim complex can be written as 
	$$0\to   {\rm Sym}_{1}((R^m)^{\ast})\stackrel{M(\bf f)^t}\lar\bigwedge^{m+1}R^{m+2}\stackrel{\partial_2'}\lar R^{m+2}\stackrel{M({\bf f})} \lar R^m.$$
Consequently, the minimal graded free resolution of $J_F$ is
	$$0\to \bigoplus R(-2D+d_j-1) \stackrel{M(\bf f)^t}\lar R(-2D)^3\oplus R(-(2D-1))^{m-1}\stackrel{\partial_2}\lar R(-D)^{3}\to J_{F}\to 0,$$
	where	$D=\deg F-1$, and we may take the syzygies of degree $2D$ to be generated by the Koszul relations of the partial derivatives of $F$.
\end{Example}

\section{Arbitrary forms}\label{Section4}
A natural question arises as to whether the RTY-problem has an inductive partition-like facet.
Namely, given a set of forms $\mathbf{f}:=\{f_1,\ldots,f_m\}$, can we aproach the problem by partitioning the set into, say, subsets $\{f_1,\ldots, f_j\}$ and $\{f_{j+1},\ldots, f_m\}$?
If $j=1$ the idea leads us naturally to consider the case of two forms only, with $f:=f_1$ and $g:=f_2 \cdots f_m$.
The first subsection below takes a look at this situation.

Throughout, as previously, $R:=\K[x_1,\ldots, x_n]$ denotes a standard graded polynomial ring over a field $\K$.

\subsection{Two forms}

We give it a start with the following
\begin{Lemma}\label{two_forms_lemma}
	Let $f,g\in R$ be non-vanishing relatively prime forms of respective degrees $d\leq e$, {\rm (}$d\geq 2${\rm )}.
Set $F:=fg$.
Then$:$
\begin{enumerate}
\item[{\rm (a)}]  If ${\rm char}(\K)$ does not divide either $e$ or $d+e$, then ${\rm indeg}(\syz (J_F))\geq {\rm indeg}(\syz (J_g)).$
\item[{\rm (b)}] If $g$ is smooth then ${\rm indeg}(\syz (J_F))\geq e-1\geq  \displaystyle\left\lfloor \frac{d+e}{2}\right\rfloor -1$.
\item[{\rm (c)}]  If $f$ is smooth then $g^2\in J_F^{\rm sat}$.
\end{enumerate}
\end{Lemma}
\begin{proof} (a)
One has:
	\begin{equation}\label{partial_sums_1}
		F_{x_i}=f_{x_i}\cdot g+f\cdot g_{x_i},
	\end{equation}
	for every $1\leq i\leq n$.
	
	Let $(s_1\cdots s_n)^t\in \syz (J_F)$ be a nonzero syzygy.
	By (\ref{partial_sums_1}), one has $$f(s_1g_{x_1}+\cdots+s_ng_{x_n})+g(s_1f_{x_1}+\cdots+s_nf_{x_n})=0.$$
	Since $\gcd(f,g)=1$, then
	\begin{equation}\label{derlog_relations_G}
		s_1g_{x_1}+\cdots+s_ng_{x_n}=h\cdot g,
	\end{equation}
	for a suitable form $h\in R$ of degree $\deg s_i-1$.
	Thus, $(s_1\cdots s_n)^t\in \mathrm{Derlog}(g)= \mathrm{Syz}(J_g)\oplus R\theta_E$, where $\theta_E:=x_1\partial_{x_1}+\cdots+x_n\partial_{x_n}$, the Euler vector.
	Therefore, $\displaystyle (s_1-\frac{1}{e} h x_1,\ldots, s_n-\frac{1}{e} h x_n)^t\in \mathrm{Syz}(J_g)$.
	Now, if some $s_i-\frac{1}{e} h x_i$ does not vanish, then $\deg(s_i)=\deg (s_i-\frac{1}{e} h x_i)$, and hence $\deg(s_i)\geq {\rm indeg}(\syz (J_g))$, for all $i$, thus proving the statement.
	Else,  $s_i=\frac{1}{e} h x_i$ for all $i$.
	Since $h\neq 0$ in this situation, this means that $(x_1\cdots x_n)^t\in \syz (J_F)$. But since $d+e\neq 0$ by assumption, then the Euler relation implies that $F=0$ -- an absurd.
	
(b)	For a complete intersection, the Koszul relations generate the first syzygies. The  inequality is then obvious since $e=\left\lfloor \frac{2e}{2}\right\rfloor\geq  \displaystyle\left\lfloor \frac{d+e}{2}\right\rfloor$.

(c)  By (\ref{partial_sums_1}), multiplying through by $g$ yields $f_{x_i}g^2\in J_F$ for every $1\leq i\leq n$.
Since $f$ is a smooth form of degree $d\geq 2$, $J_f$ is $\fm$-primary, and therefore $g^2\in (J_F)^{\rm sat}$.
\end{proof}

\begin{Proposition}\label{two_forms_all}
	{\rm (${\rm char}(\K)$ does not divide any of $d,e,d+e$)}
	Let $f,g\in R$ be non-vanishing relatively prime forms of respective degrees $d\leq e$, {\rm (}$d\geq 2${\rm )} Suppose that$:$
	\begin{enumerate}
		\item[{\rm (i)}]  $f$ is smooth.
		\item[{\rm (ii)}] For every prime ideal $\wp\supset \langle f,g\rangle$ of height $n-1$, the Jacobian matrix of $\langle f,g\rangle$ has maximal rank over $R/\wp$.
	\end{enumerate}		
		Then $\{f,g\}$ satisfy the {\rm RTY}-property.
\end{Proposition}
\begin{proof}
 Let $\Theta$ denote the Jacobian matrix of $\langle f,g \rangle$. Since $\{f,g\}$ is a regular sequence, by Proposition~\ref{heights_main} and Theorem~\ref{main_general_forms}, it suffices to  prove that the  homogeneous ideal $\mathbb J:=\langle  I_2(\Theta), f\rangle$ is $\fm$-primary.

Suppose otherwise and pick a prime ideal $\mathbb J\subset\wp\subset R$ of height at most $n-1$.
Then, since $f\in \wp$ is assumed to be smooth, there is a partial derivative, say, $f_{x_1}$ not belonging to $\wp$.
At the other end,  $g_{x_i}f_{x_1}-g_{x_1}f_{x_i}\in \mathbb J\subset \wp$, for every $i$.
Using the Euler relation of $g$, we can write
{\small
	\begin{eqnarray*}
		egf_{x_1}&=&x_1f_{x_1}g_{x_1}+x_2f_{x_1}g_{x_2}+\cdots + x_nf_{x_1}g_{x_n}\\
		&=& x_1f_{x_1}g_{x_1}+ [x_2(f_{x_1}g_{x_2}-f_{x_2}g_{x_1})+ x_2f_{x_2}g_{x_1}]+\cdots +[x_n(f_{x_1}g_{x_n}-f_{x_n}g_{x_1}) +x_nf_{x_n}g_{x_1}]\\
		&=& g_{x_1}(x_1f_{x_1}+\cdots + x_nf_{x_n}) + x_2(f_{x_1}g_{x_2}-f_{x_2}g_{x_1}) + \cdots + x_n(f_{x_1}g_{x_n}-f_{x_n}g_{x_1})\\
		&=& dg_{x_1}f + x_2(f_{x_1}g_{x_2}-f_{x_2}g_{x_1}) + \cdots + x_n(f_{x_1}g_{x_n}-f_{x_n}g_{x_1}).
	\end{eqnarray*}
}
Clearly, the last expression belongs to $\mathbb J\subset \wp.$ Since $f_{x_1}\notin \wp$ then $g\in\wp$.
We thus conclude that $\langle f,g\rangle\subset \wp$.
For such a prime, we can apply the main standing hypothesis.
Namely, note that the following matrix product
$$\left[\begin{matrix}
	-g_{x_1} & f_{x_1} 
\end{matrix}
\right]
\Theta= 
\left[\begin{matrix}
	0 & f_{x_1}g_{x_2}-f_{x_2}g_{x_1} & \dots & f_{x_1}g_{x_n}-f_{x_n}g_{x_1}
\end{matrix}
\right]
$$
vanishes modulo $\wp$ because $I_2(\Theta)\subset \wp$.
But, by the standing condition of this item,  $\Theta$ has maximal rank (i. e., $2$) over $\wp$. Therefore, $f_{x_1} \in \wp$, since $\wp$ is prime; a contradiction.
\end{proof}
\begin{Remark}\rm
One can give an alternative proof of the above proposition, by noting that $g\mathbb J\subset J_F$. Indeed, again by (\ref{partial_sums_1}) one obtains the relations
$$g_{x_i}F_{x_j}-g_{x_j}F_{x_i}=(g_{x_i}f_{x_j}-g_{x_j}f_{x_i})g,$$
for $1\leq i<j\leq n$.
Therefore, $g\,I_2(\Theta)\subset J_F$.
At the other end, $gf=F\in J_F$ by the Euler relation of $F$.
This shows the above inclusion.
Now, collecting this assertion and the argument in the proposition, one gets $g\in (J_F)^{\rm sat}$.
Clearly, $g\notin J_F$ for degree reasons, since ${\rm indeg}(J_F)=e+d-1$.
Therefore, $\fm$ is an associated prime of $R/J_F$, as was to be shown.
\end{Remark}

\begin{Corollary}\label{theorem_planar} Let $f, g \in R:=\mathbb K[x,y,z]$ with $\deg(f)\geq 1$, $\deg(g)\geq 2$, and $\gcd(f,g)=1$. 
	Suppose that $V(f)\subset \mathbb{P}_{\K}^2$ is smooth, $V(g)\subset \mathbb{P}_{\K}^2$ is not a cone, $V(f)$ and $V(g)$ are nowhere tangent.
	Then $\{f,g\}$ satisfy the {\rm RTY}-property. {\rm (}Hence, $F=fg$ is not a free divisor{\rm )}.
\end{Corollary}
Note that, in this corollary, since $V(g)$ is not required to be locally smooth at the minimal primes of $\langle f,g \rangle$, the hypotheses are weaker than transversal intersection.

The next result imposes no restriction in how the two corresponding divisors intersect. 

\begin{Theorem}\label{two__forms_one_smooth} {\rm ($\mathrm{char}(\K)\nmid d$)} Let $f,g\in R:=\mathbb K[x_1,\ldots,x_n]$ be non-vanishing relatively prime forms. 	Suppose that $f$ defines a smooth hypersurface of degree $d$, and in addition, that one of the following conditions is satisfied$:$
	\begin{itemize}
		\item[(i)] $\deg g=\deg f=d\geq 2$ and ${\rm indeg}(\syz (J_g))\geq 2$.
		\item[(ii)] $g\notin J_f$.
	\end{itemize}
	Then $\{f,g\}$ satisfy the {\rm RTY}-property. 
\end{Theorem}
\begin{proof} Set $F:=fg$ and $\fm:=\langle x_1,\ldots,x_n\rangle$.
	
We will show that, under either (i) or (ii), there exists a power of $g$, not belonging to $J_F$, driving a power of $\fm$ into $J_F$. This clearly proves the assertion.
	
First, since $f$ is smooth, Lemma~\ref{two_forms_lemma} (c) implies that $g^2\mathfrak{m}^{r}\subseteq J_F$, for a suitable integer $r>0$.
	
	It remains to show that, under either of the assumptions (i) and (ii), one has $g^2\notin J_F$.
	
	Suppose that $g^2\in J_F$ and let us show that neither (i) nor (ii) holds. Indeed, this inclusion implies that $\deg(g)\geq \deg(f)-1$, and there are forms $h_1,\ldots,h_n$ in $R$ of degree $\deg(g)-\deg(f)+1$ such that
	\begin{equation}\label{equal}
		g^2=h_1F_{x_1}+\cdots+h_nF_{x_n}
		=g(\sum_{i=1}^{n} h_i f_{x_i}) + f(\sum_{i=1}^{n}h_ig_{x_i}).
	\end{equation}
	
	Since $\gcd(f,g)=1$, the above yields  relations
	
	\begin{itemize}
		\item[(a)] $\sum_{i=1}^{n}h_ig_{x_i}=Q\cdot g$, and
		\item[(b)] $g=Q\cdot f+\sum_{i=1}^nh_i f_{x_i}$,
	\end{itemize}
	 for a suitable form  $Q\in R$ of degree $\deg(g)-\deg(f)$.

	Suppose (i) is satisfied. Then, $\deg(g)-\deg(f)+1=1$, giving that $h_1,\ldots, h_n$ are linear forms, and $Q=:\alpha\in\mathbb K$, a constant. By (a) and from Euler's relation, we obtain a linear syzygy of $J_g$
	$$\sum_{i=1}^{n}\left(h_i-\frac{\alpha}{d} x_i\right)g_{x_i}=0.$$ This contradicts assumption (i), unless this syzygy vanishes, in which  case, one would have $h_i=(\alpha/d) x_i$ for every $i$. Substituting in (\ref{equal}) above yields, after a simple calculation: $g^2=2\alpha F=2\alpha fg$. Therefore, either $g=2\alpha f$, with $\alpha\neq 0$, or else $g=0$ -- both contradicting our basic assumptions. 
	
	Suppose (ii) is satisfied. By (b) above and Euler's relation for $f$ (as $\mathrm{char}(\K)\nmid d$ by assumption) one gets $g\in J_f$ --  a contradiction. 
\end{proof}

Let us  look at some consequences of  Theorem~\ref{two__forms_one_smooth}. 

\begin{Corollary}\label{two_smooth_forms}
	{\rm ($\mathrm{char}(\K)\nmid d$)}
	Let $f,g\in R:=\mathbb K[x_1,\ldots,x_n]$ be distinct smooth forms of the same degree $d\geq 3$.
	Then $\{f,g\}$ satisfy the {\rm RTY}-property.	
\end{Corollary}
\begin{proof}
	Since $d\geq 3$ and $g$ is smooth, then ${\rm indeg}(\syz (J_g))\geq 2$, hence Theorem~\ref{two__forms_one_smooth} applies.
\end{proof}

Next is a sample of the induction/partition principle mentioned at the beginning of the section.

\begin{Corollary} \label{three_forms} Let $f_1,f_2,f_3\in R:=\mathbb K[x_1,\ldots,x_n]$ be  forms satisfying the following conditions:
	\begin{enumerate}
		\item[{\rm (1)}]  Any two $f_i,f_j$ $(i\neq j$) are relatively prime.
		\item[{\rm (2)}] 	Degrees restriction$:$ $\deg(f_2)\geq\deg(f_1)\geq 2$,  $\deg(f_1)+\deg(f_2)=\deg(f_3)\geq 5$.
		\item[{\rm (3)}]  $f_2$ and $f_3$ are smooth.
	\end{enumerate}
	Then, the set $\{f_1,f_2,f_3\}$ satisfies the {\rm RTY}-property.
\end{Corollary}
\begin{proof} Set $g:=f_1f_2$ and $f:=f_3$. Then $\deg(g)=\deg(f) \geq5$.
	
	Since $f_2$ is smooth, and of degree $\geq 3$ because of (2),  Lemma~\ref{two_forms_lemma} (i) implies that ${\rm indeg}(\syz (J_g))\geq \deg(f_2)-1\geq 2$.
	
	By Theorem~\ref{two__forms_one_smooth}  (i), $\{f,g\}$ satisfies the {\rm RTY}-property, hence so does $\{f_1,f_2,f_3\}$.
\end{proof}
The following example shows the relevance of the degree assumptions in the last corollary.

\begin{Example}\rm
Let $R=\K[x_1,\ldots,x_5]$, with $f_1=x_1^2+x_2x_3$, $f_2=x_1^2+x_2^2+x_2x_3$, $f_3=x_5^5-x_4f_1f_2$.
By Proposition~\ref{free-again} (iv), the form $f_1f_2$ is a free divisor and involves but the first three variables.
Then $F=f_1f_2f_3$ is again a free divisor by \cite[Theorem 2.12]{SimStef2014}, thus implying the failure of the corollary by as much as possible.
The syzygy matrix of $J_F$ is quite recursive:
\begin{equation}\label{matrix_of_Jf}
	\syz (J_F):=\left(
	\begin{array}{ccc@{\quad\vrule\quad}cc}
		&& & x_1 & 0\\[3pt]
		& \raise5pt\hbox{$\syz (J_{f_1f_2})$} & & \raise5pt\hbox{$\vdots$} &\raise5pt\hbox{$\vdots$}\\
		&& & \raise5pt\hbox{$x_{3}$} & \raise5pt\hbox{$0$}\\
		\multispan5\hrulefill\\[0.5pt]
		0&\cdots & 0 & -8x_4 & 5x_5^4\\[5pt]
		0&\cdots & 0 & -4/5x_5 &f_1f_2
	\end{array}
	\right).
\end{equation}
\end{Example}

We don't know if the  assumption on the initial degree of the syzygies of $J_g$ in Theorem~\ref{two__forms_one_smooth} (i) is essential in general for every given $d\geq 3$.
A good testing ground might be the higher versions of a cuspidal plane curve, however some care is to be exercised due to \cite[Proposition 2.11 and Theorem 2.12]{SimStef2014}, as displayed in the above example.

A {\em Fermat form of degree $d$} is a homogeneous polynomial  $f:=c_1x_1^d+\cdots+c_nx_n^d \in R:=\mathbb K[x_1,\ldots,x_n]$, where not all the coefficients $c_1,\ldots,c_n\in \mathbb K$ vanish.  Let us note that all the results of Subsection~\ref{Jacobian_general} go through as well  for general Fermat forms as defined in terms of Zariski open sets in the respective product parameter space of the $n$th dimensional vector spaces spanned by the pure powers $x_1^{d_i},\ldots,x_n^{d_i} (1\leq d_1\leq\cdots\leq d_m)$.
Consequently, the RTY-property holds true for a set of general Fermat forms.

The next result deals instead with not necessarily general such forms.
If none of the coefficients $c_i$ vanishes, one says that the Fermat form $f$ has {\em full rank}, which in the suitable characteristic amounts to saying that $f$ is smooth.

\begin{Corollary}\label{prop_Fermat_forms} {\rm (char$(\K)\nmid d_m$)} Let ${\bf f}:=\{f_1,\ldots,f_m\}\subset R, m\leq n+1$, be a set of Fermat forms of degrees $1\leq \deg(f_1)\leq\cdots\leq\deg(f_{m-1}) \leq \deg(f_{m})$, satisfying the following conditions:
\begin{enumerate}
	\item[{\rm (1)}]  $\deg f_m\geq \deg f_{m-1}-2$.
	\item[{\rm (2)}]  Any two $f_i,f_j$ ($i\neq j$) are relatively prime.
	\item[{\rm (3)}]  $f_m$ has full rank.
	\item[{\rm (4)}]  For some set of  indices $\{i_1,\ldots,i_{m-1}\}\subset \{1,\ldots,n\}$
	the coefficient of the monomial term $\mathbf{m}:=x_{i_1}^{d_1}\cdots x_{i_{m-1}}^{d_{m-1}}$ of the product $f_1\cdots f_{m-1}$ does not vanish.
\end{enumerate}
Then ${\bf f}$ satisfies the {\rm RTY}-property.
\end{Corollary}
\begin{proof} Set $f:=f_m$,  $g:=f_1\cdots f_{m-1}$, and $d_i:=\deg f_i$. 
Note that $f$ is smooth, since $f$ has full rank and  char$(\K)\nmid d_m$; in particular, $J_f=\langle x_1^{d_m\, -1},\ldots,x_n^{d_m\, -1}\rangle$. 

It then suffices to show that condition (ii) in Theorem~\ref{two__forms_one_smooth} is satisfied, namely, that $g\notin J_f$.
As $J_f$ is a monomial ideal, $g\in J_f \Rightarrow \mathbf{m}\in J_f$. But from the degrees condition, this is impossible. 
\end{proof}

\begin{Remark}\rm
As a side, the RTY-property of Fermat forms $\mathbf{f}$ satisfying the conditions as in Corollary~\ref{prop_Fermat_forms}  does not necessarily imply that the ideal $\langle I_m(\Theta(\mathbf{f})), \mathbf{f}\rangle$ has maximal height. A simplest example is as follows, with $m=2$, $d_1=2,d_2=5$: $f_1=x^2, f_2=x^5+y^5+z^5 $.
Here $\Ht  \langle I_2(\Theta(f_1,f_2), f_1,f_2\rangle=2$, which is actually minimal possible. At the other end, full rank alone (i.e., smoothness) does not do either, as is the case of taking $f_1=x^2+y^2+z^2$. Thus, it would seem that the results of Subsection~\ref{Jacobian_general}, as translated to Fermat forms, is genuinely about general such forms.
\end{Remark}

To wrap up the present section on the case of two forms we next discuss the details of  two such forms of degrees $2$ in $R=\K[x,y,z]$.

\subsection{Two ternary smooth quadrics} \label{Subsection4.2}

In this part we focus on the case where $d=2$, i.e., we discuss in detail the RTY-property of a quartic in $\K[x,y,z]$ which is the product of two smooth quadrics.
One reason to consider this environment is that no irreducible plane projective curve  of degree $2\leq d+1 \leq 4$ is a free divisor.
For a rough classification of the various types we resort to the recent theory of a Bourbaki degree of a plane projective curve as developed in \cite{Bour2024}.
Our approach is driven toward understanding the restriction to having the RTY-property, equivalently in dimension three, to capture conditions for freeness. 

Quite generally, for any form $F\in R=\K[x,y,z]$ of degree $\geq 2$, we assume that its partial derivatives are algebraically independent over $\K$, as otherwise in this dimension the Hesse classical result forces them to actually be $k$-linearly dependent. This hypothesis can be rephrased, e.g., to the effect that $k$-subalgebra of $R$ generated by the partial derivatives has Krull dimension three, or that the  fiber cone algebra of $J_F$ is (isomorphic to) $\K[x,y,z]$.

\subsubsection{Free divisor characterization}

We list a couple of additional characterizations of a free divisor $F$ in the case $F$ is the product of two non-degenerate quadrics.

\begin{Proposition}\label{free-again}
Let $F\in R=\K[x,y,z]$ be the product of two distinct non-degenerate quadrics.
The following conditions are equivalent$:$
\begin{enumerate}
		\item[{ \rm (i)}]  $J_F$ is a codimension two perfect ideal with minimal graded free resolution
		$$0\rar  R(-4) \oplus R(-5) \lar R(-3)^3 \lar R \lar R/J_F \rar 0.$$
	\item[{ \rm (ii)}]  ${\rm indeg}\, {\rm Syz}(J_F)=1$ and $\deg {\rm Syz}(J_F)=7$.	
		\item[{ \rm (iii)}]  $F$ is a free divisor.
			\item[{ \rm (iv)}]  {\rm (char $(\K)=0$)} Up to a change of variables,  $F=(x^2+yz)(x^2+vy^2+yz)$, for a certain  nonzero $v\in \K$.
\end{enumerate}
\end{Proposition}
\begin{proof}
(i)  $\Rightarrow$ (ii).
The first assertion is obvious. The second assertion follows from the well-known calculation of the degree (multiplicity) via the Hilbert series off the free resolution.

(ii)  $\Rightarrow$ (iii).
Perhaps fastest, since $V(F)$ is singular, is to show that the Bourbaki degree of $V(F)$ vanishes by use of \cite[Theorem 2.1 (a) and Remark 2.5]{Bour2024}.

(iii) $\Rightarrow$ (iv).
Due to Corollary~\ref{theorem_planar}, we know that the given quadrics must have a tangential intersection.
Say, they intersect tangentially at the point $(0:0:1)$, corresponding to the coordinate ideal $(x,y)\subset R$.
Now, change variables to assume that one of the two quadrics has defining form $f:=x^2+yz$. The tangential assumption leads us to write the other quadric in the form $g:=tx^2+uxy+vy^2+yz$, for scalars $t,u,v$ in $\K$, with $t\neq 0.$

Since $f$ and $g$ are relatively prime, given any syzygy $[A\;B\; C]^t$ of $J_F$ is equivalent to giving the relations
\begin{equation}\label{basic_relations}
	\left\{\begin{array}{cccc}
	Ag_x+Bg_y+Cg_z&=&\alpha g \\
	Af_x+Bf_y+Cf_z&=&-\alpha f,  
\end{array}\right.
\end{equation}
for some $ \alpha\in \mathbb{K}$.

We know that $F$ is a homogeneous free divisor if and only if $J_F$ is a codimension two perfect ideal. Since $F$ has degree $4$, this is also equivalent to having
\begin{equation}\label{syzygies_vectorspaces}
	\dim_{\K} \syz (J_F)_1=\dim_{\K}\syz (J_F)_2/R_1\,\syz (J_F)_1=1,
\end{equation}
where $ \syz (J_F)_i$ denotes the vector space of $ \syz (J_F)$ spanned in degree $i$, and $R_1\,\syz (J_F)_1$ denotes the subspace of $ \syz (J_F)_2$ spanned by the elements of the form $a\mathcal{S}$, with $a\in R_1$ and $\mathcal{S}\in \syz (J_F)_1$. 

Let then $[A\;B\; C]^t\in \syz( J_F)$ denote the (up to scalars) unique syzygy of degree one.

\medskip

{\bf Claim 1.} Either $t=1$, $u=0$ and $v\neq 0$, or else $u^2=4v(t-1)$ with $t\neq 1$.

To see this, write $A=a_1x+a_2y+a_3z, B=b_1x+b_2y+b_3z, C=c_1x+c_2y+c_3z.$
Matching monomials in $x,y,z$ on the two sides of (\ref{basic_relations}) yields the following  equalities:
\begin{equation}\label{coeff_relations} a_1=-\frac{1}{4}a_3u,~b_1=-2a_3,~c_1=-2a_2,~c_3=-\frac{1}{2}a_3u-b_2,~ b_3=c_2=0; \; \alpha=\frac{1}{2}a_3u,
	\end{equation} 
and vanishing relations  which we choose to display in matrix format:    

	\begin{eqnarray}
	&&\underbrace{\left[\begin{array}{ccccccccc}
			0&t-1&0\\
			0&-\frac{1}{2}ut-u&0\\
			2(t-1)&-4v-\frac{3}{4}u^2&u\\
			u&-\frac{1}{2}uv&2v
		\end{array}\right]}_{\mathcal{M}}
		\, \left[\begin{array}{c}
		a_2\\
		a_3\\
		b_2
	\end{array}\right]=\left[\begin{array}{c}
	0\\
	0\\
	0\\
	0
	\end{array}\right]
\end{eqnarray}
Now, by (\ref{coeff_relations}), $a_2,a_3,b_2$ can't all vanish.
Therefore, $\mathcal{M}$ has rank at most two, i. e., the ideal $I_3(\mathcal{M})=\langle 	(t-1)(u^2-4v(t-1)),\,	u(\frac{t+2}{2})(u^2-4v(t-1)) \rangle$ vanishes.
Therefore, either $t=1$ and $u=0$, or else $u^2=4v(t-1)$ with $t\neq 1$.
Moreover, if $v=0$ in the first case it would say that $f=g$ against our assumption that $f$ and $g$ are relatively prime. 

This wraps up the contents of the above claim.

\medskip

{\bf Claim 2.}  The equalities in (\ref{syzygies_vectorspaces}) are equivalent to having $t=1, u=0, v\neq 0$.

By the previous claim, it suffices to discuss the second equality in (\ref{syzygies_vectorspaces}).

We first show that the values $t=1, u=0, v\neq 0$ imply an element of $[A\,B\,C]^t$ of $\syz (J_F)_2/R_1\,\syz (J_F)_1$.
It suffices to show that such a guessed candidate satisfies (\ref{basic_relations}), this time around with $\alpha\in R_1$.
We take 
$$[A\,B\,C]^t:=[xz; -(2x^2+vy^2+2yz); 2vx^2+3vyz+2z^2)]^t
\in (R^3)_2,$$
(where the separators have been written for clarity) whose shape comes out of hand-calculation or computer assistance.

Straightforward calculation shows that $[A\,B\,C]^t$ satisfies (\ref{basic_relations}), with $\alpha=2vy\in R$.
Since none of the factors $x$ or $z$ of its first coordinate divides the other coordinates, then $[A\,B\,C]^t\in\syz (J_F)_2/R_1\,\syz (J_F)_1$, as required.

Conversely, let  $[A\,B\,C]^t\in \syz (J_F)_2$. 
Firstly, as before, it satisfies the relations in (\ref{basic_relations}),  with $\alpha\in R_1$.

If we contradict the values $t=1, u=0, v\neq 0$ then, by Claim 1, we have $u^2=4v(t-1)$ with $t\neq 1$.
Note that $t$ does not vanish since $g$ is irreducible, and that $u=0 \Leftrightarrow v=0$.

We claim that, in this situation, necessarily $[A\,B\,C]^t\in R_1\, \syz (J_F)_1$.

Write $a_i, b_i, c_i (1\leq i\leq 6)$ for the respective coefficients of $A,B,C$.
As before, matching monomials in $x,y,z$ on the two sides of (\ref{basic_relations}) (now with $\alpha=\lambda_1x+\lambda_2y+\lambda_3z \in R_1$) gives rise to the following relations:
$$\lambda_1=-2{a}_1,~\lambda_2=-2{a}_2-{c}_1,~\lambda_3=-2{a}_3-{b}_1,~{b}_3=-2{a}_6,~{c}_2=-2{a}_4,~{b}_6={c}_4=0,$$ and   
$$\begin{array}{cccccccccc}
	3u{a}_1+4t{a}_2+2v{b}_1+u{b}_2+(t+1){c}_1&=0,&4t{a}_1+u{b}_1&=0, \\
	2v{a}_1+3u{a}_2+2(t-1){a}_4+2v{b}_2+u{b}_4+u{c}_1&=0,& (t-1){a}_6&=0,                                                                  \\
	2v{a}_2+u{a}_4+2v{b}_4+v{c}_1&=0,&2{a}_3+u{a}_6+{b}_1+{b}_5+{c}_6&=0, \\
	4t{a}_3-2u{a}_6+(1+t){b}_1&=0,&-2{a}_1+2{a}_5+{b}_2+{c}_3&=0, \\
	2{a}_1+3u{a}_3+2t{a}_5-4v{a}_6+u{b}_1+{b}_2+u{b}_5+{c}_3&=0,&-2{a}_2+{b}_4-{c}_1+{c}_5&=0,\\
	2{a}_2+2v{a}_3+u{a}_5+v{b}_1+{b}_4+2v{b}_5+{c}_1+{c}_5&=0,& -2{a}_3-{b}_1+{b}_5+{c}_6&=0.\\
\end{array}$$	

Let us separate in two cases, both with  $t\notin \{0,1\}$.

\smallskip

$\bullet$  Case 1: $u=v=0$.

Thus, $g=tx^2+yz$, with $t\neq 0,1$.
Here 
$$\partial F/\partial  y=(1+t)x^2z+2yz^2,\; \partial F/\partial  z=(1+t)x^2y+2y^2z,$$
yielding the syzygy $[0 -y\,\, z]^t\in \syz (J_F)_1$.

In addition, substituting $u=v=0$ above yields $A=0, c_1=0,  {b}_1=0, {c}_3=-{b}_2,~ {c}_5=-{b}_4,~{c}_6=-{b}_5.$
Therefore, $[A\,B\,C]^t=(b_2x+b_4y+b_5z)[0 -y\,\, z]^t\in R_1\syz (J_F)_1.$

\medskip

$\bullet$  Case 2: $u\neq 0$ and $v=\frac{u^2}{4(t-1)}.$ 

In this case, we have ${a}_1={a}_3={a}_6=	{b}_1=0$, and the following additional relations
\begin{eqnarray*}
&&{b}_2=-\frac{2(t-1)}{u}{a}_2,~
{b}_4=-\frac{2(t-1)}{u}{a}_4,~
{b}_5=-\frac{2(t-1)}{u}{a}_5,\\
&&	{c}_1=-2{a}_2, ~~
{c}_3=\frac{2(t-1)}{u}{a}_2-2{a}_5, ~~
{c}_5=\frac{2(t-1)}{u}{a}_4,~~
{c}_6=\frac{2(t-1)}{u}{a}_5.
\end{eqnarray*} 
Therefore, 
$$[A\,B\,C]^t=(a_2x+a_4y+a_5z)[y\; -(2(t-1)/u)y\;-2x+(2(t-1)/u)z]^t\in  R_1\syz (J_F)_1.$$

\smallskip

(iv)  $\Rightarrow$ (i).
From (\ref{syzygies_vectorspaces}), as obtained from the argument of the previous implication, $\syz (J_F)$ has exactly one essential generator in either degrees $1$ or $2$.
Since $\deg F=4=(1+2)+1$, it is well-known as pointed out before that no other independent minimal generator exists.
This implies the desired free resolution shape.
\end{proof}

\subsubsection{Relation to the Bourbaki degree}
We give a rough classification of the cases where $F$ as in the previous part satisfies the RTY-property in terms of the Bourbaki degree of $F$ (\cite{Bour2024}).
The latter is defined as the co-degree of the Bourbaki ideal of ${\rm Syz}(J_F)$ with respect to a minimal generator $e$ of standard initial degree. 
Note that, quite generally,  if $F$ is a form of degree four then $e\leq 3$, the natural upper bound off the Koszul syzygies.

In the following list $B(F)$ denotes the Bourbaki degree of $F$.
The terminology regarding special classes of curves is due to Dimca and Sticlaru (see \cite{Dimca-Sticlaru2018}, \cite{Dimca-Sticlaru2020}).

It will save space and look handier to state the results as in the following table.

\begin{center}
	\begin{tabular}{|c|c|}
		\hline
		{\bf DATA}  &$\kern-4pt e \hspace*{1.6cm}|\hspace*{1.5cm}  \deg J_F \hspace*{1.55cm}| \hspace*{1.5cm} B(F)$\\
		\hline
		{\bf Transversal/General}  &$2 \hspace*{1.52cm}|\hspace*{1.9cm}  4 \hspace*{2.08cm}| \hspace*{1.8cm} 3\hspace*{.45cm}$\\
		\hline
		&\\
		\hspace{.23cm} Free resolution \hspace{.23cm} &
		$0\rar R(-7)^2 \lar R(-5) \oplus R(-6)^3 \lar R(-3)^3 \lar R \lar R/J_F \rar 0$\\
		\hline
		&\\
		\hspace{.23cm} Resolution details \hspace{.23cm} & Koszul syzygies in degree $6$\\
		& Transposed Tjurina forms of $f,g$ as columns in degree $7$	\\
		&\\
		\hline 
		\hspace{.23cm} Example \hspace{.23cm} & 	$f=x^2+yz, g=y^2+xz$\\
		\hline
		
		\hline
		{\bf $\bf 3$-syzygy curve}  &$2 \hspace*{1.52cm}|\hspace*{1.9cm}  5 \hspace*{2.08cm}| \hspace*{1.8cm} 2\hspace*{.45cm}$\\
		\hline
		&\\
		\hspace{.23cm} Free resolution \hspace{.23cm} &
		$0\rar R(-7) \lar R(-5)^2 \oplus R(-6) \lar R(-3)^3 \lar R \lar R/J_F \rar 0$\\
		&\\
		\hline
		\hspace{.23cm} Example \hspace{.23cm} & 	$f=x^2+yz, g=x^2+y^2-yz$\\
		\hline
		
		\hline
		{\bf Plus-one curve}  &$2 \hspace*{1.52cm}|\hspace*{1.9cm}  6 \hspace*{2.08cm}| \hspace*{1.8cm} 1\hspace*{.45cm}$\\
		\hline
		&\\
		\hspace{.23cm} Free resolution \hspace{.23cm} &
		$0\rar R(-6) \lar  R(-5)^3 \lar R(-3)^3 \lar R \lar R/J_F \rar 0$\\
		&\\
		\hline
		\hspace{.23cm} Example \hspace{.23cm} & 	$f=x^2+yz, g=x^2+xy+y^2+yz$\\
		\hline
		
		\hline
		{\bf Nearly free}  &$1 \hspace*{1.52cm}|\hspace*{1.9cm}  6 \hspace*{2.08cm}| \hspace*{1.8cm} 1\hspace*{.45cm}$\\
		\hline
		&\\
		\hspace{.23cm} Free resolution \hspace{.23cm} &
		$0\rar R(-7) \lar R(-4) \oplus R(-6)^2 \lar R(-3)^3 \lar R \lar R/J_F \rar 0$\\
		&\\
		\hline
		\hspace{.23cm} Resolution details \hspace{.23cm} & 
		The cases of non-free divisors  in the proof of Proposition~\ref{free-again}, (iv)\\
		\hline
		\hspace{.23cm} Example \hspace{.23cm} & 
		$f=x^2+yz, g=x^2-yz$ \\
		
		\hline
		{\bf  Free divisor} &$1 \hspace*{1.52cm}|\hspace*{1.9cm}  7 \hspace*{2.08cm}| \hspace*{1.8cm} 0\hspace*{.45cm}$\\
		\hline
		&\\
		\hspace{.23cm} Free resolution \hspace{.23cm} &
		$0\rar  R(-4) \oplus R(-5) \lar R(-3)^3 \lar R \lar R/J_F \rar 0$\\
		&\\
		\hline
		\hspace{.23cm} Example \hspace{.23cm} & 	
		$f=x^2+yz, g=x^2+y^2+yz$\\
		\hline
	\end{tabular}
\end{center}

\smallskip

\begin{Caveat}\rm
	The following complementary information may help above:
	
	$\bullet$ For Tranversal/General: $B(F)=3$ by the formulary in \cite[Section 2]{Bour2024}.
	
	$\bullet$ If $F$ is not transversal then $\deg R/J_F\geq 5$. 
	
	$\bullet$ For Nearly free:  $\deg R/J_F\geq 6$ in virtue of the theorem of Du Plessis--Wall (see \cite[Theorem 3.3]{Bour2024}). Moreover, the cases of non-free divisors discussed in the proof of Proposition~\ref{free-again}, (iv) fall under this category.
\end{Caveat}

\begin{Question}\rm The free case above has a unique singular point of multiplicity $7$. This raises the question as to the distribution of the singular points in the free divisor environment.
\end{Question}

\subsection{Higher number of forms}

Let $f_1,\ldots,f_m\in R:=\mathbb K[x_1,\ldots,x_n], m\geq 2$, be forms of degrees $\deg f_i=d_i\geq 2, i=1,\ldots,m$.
 With adapted conditions, a simulacrum of Corollary~\ref{three_forms} might be available.
 To work out the ideas, we assume throughout that every $f_i$ is reduced, and that any two $f_i,f_j$ ($i\neq j$) are relatively prime.

As before, set $F:=f_1f_2\cdots f_m$, and let $J_F\subset R$ be the gradient ideal of $F$.
Note that $J_F$ is contained in the ideal generated by the $(m-1)$-fold products of $\{f_1,\ldots,f_m\}$.
In particular, $J_F$ has codimension at most $2$. But, since $F$ is reduced, its codimension is exactly $2$.
Finally, observe that, though the given forms may have different degrees $d_1,\ldots,d_m$, the ideal $J_F$ is equigenerated in degree $d_1+\cdots+d_m-1$.

We note the following degree bound for the Jacobian syzygies of $J$, generalizing the opening result of Lemma~\ref{two_forms_lemma}.

\begin{Lemma}\label{syz_deg_product_forms}
	Let $\{f_1,\ldots,f_m\}\subset R:=\mathbb K[x_1,\ldots,x_n], m\geq 2$, be  forms of degrees $2\leq  \deg f_1=d_1\leq \cdots\leq \deg f_m=d_m$, such that any two of them are relatively prime.
	If $f_2,\ldots,f_m$ are smooth, then$:$
	$${\rm indeg}(\syz (J_F))\geq
	\displaystyle\left\lfloor \frac{d_1+\cdots+d_m}{m}\right\rfloor-1.$$
\end{Lemma}
\begin{proof} 
	We prove the result by induction on $m\geq 2$.
	For $m=2$, this is Lemma~\ref{two_forms_lemma} (b), with $f=f_1, g=f_2$ in the notation there, observing that $f_2$ is smooth here by assumption.
	
	For the inductive step, suppose $m\geq 3$ and the result valid for $m-1$ such forms.
	In particular, letting $G:=f_2\cdots f_m$, the inductive step gives
	\begin{equation}\label{inductive_bound}
		{\rm indeg}(\syz (J_G))\geq
		\displaystyle\left\lfloor \frac{d_2+\cdots+d_m}{m-1}\right\rfloor-1
		\geq \displaystyle\left\lfloor \frac{d_1+\cdots+d_m}{m}\right\rfloor-1,
	\end{equation}
	where the right most lower bound follows from the  inequality 
	$$(m-1)(d_1+\cdots+d_m)\leq m(d_2+\cdots+d_m).$$
	
	At the other end, $F=f_1G$, so by Lemma~\ref{two_forms_lemma} (b) once again, as applied with $f=f_1$ and $g=G$ in the notation there, one has ${\rm indeg}(\syz (J_F))\geq {\rm indeg}(\syz (J_G))$.
	In view of (\ref{inductive_bound}), we are through.
\end{proof}

\begin{Remark}\rm
	The bounds in the above lemma are quite way off the bounds in the case of sufficiently general forms. Thus, for $n=3$ and three such forms of degree $3$, one finds ${\rm indeg}(\syz (J_F))=7 \gg 2$.
\end{Remark}

For given integers $d_1\,\ldots,d_r$,  set 
$$\mathfrak{d}(d_1,\ldots,d_r):=\left\lfloor \frac{d_1+\cdots+d_r}{r}\right\rfloor-1,$$
borrowing from the lower bound above.


\medskip

\begin{Proposition}\label{unbalanced_degrees} {\rm (char $(\K)=0$)}
	Let $\{f_1,\ldots,f_m\}\subset R:=\mathbb K[x_1,\ldots,x_n], m\geq 2$, be  forms of degrees $2\leq  \deg f_1=d_1\leq \cdots\leq \deg f_m=d_m$, such that any two of them are relatively prime.
	If $f_2,\ldots,f_m$ are smooth and
	there exists $j\in \{2,\ldots,m\}$ such that $$
	\mathfrak{d}(d_1,\ldots,\widehat{d_j},\ldots, d_m)\geq
	\left(\sum_{i\neq j}d_i\right) -d_j+3,$$
	then $\{f_1,\ldots,f_m\}$ satisfy the {\rm RTY}-property.
\end{Proposition}
\begin{proof} Let $f:=f_j$ and $G:=F/f$, and set $e:=d_j=\deg(f)$,  $d:=\sum_{i\neq j} d_i=\deg(G)$, so the standing assumption says that
	\begin{equation}\label{assumption_redone}
		\mathfrak{d}(d_1,\ldots,\widehat{d_j},\ldots, d_m)\geq d-e+3.
	\end{equation}
	
	By Lemma~\ref{two_forms_lemma} (c), since $f$ is assumed to be smooth, we have $G^2\in (J_F)^{\rm sat}$. It suffices to prove that $G^2\notin J_F$.
	The argument is pretty much the same as in first part of the proof of Theorem~\ref{two__forms_one_smooth}, so we brief on the main points there.
	Namely, supposing the contrary, there will exist forms
	 $L_1,\ldots,L_n\in R_{d-e+1}$, such that eventually
	\begin{eqnarray*}
		G-(L_1f_{x_1}+\cdots+L_nf_{x_n})&=&B\cdot f\\ 
		L_1G_{x_1}+\cdots+L_nG_{x_n}&=& B\cdot G,
	\end{eqnarray*}
	for some $B\in R_{d-e}$.
	
	Euler's relation for $G$ as plugged into the second of the above relations  leads to the syzygy of degree $d-e+1$
	$$(L_1-(B/d)x_1)G_{x_1}+\cdots+(L_n-(B/d)x_n)G_{x_n}=0.$$
	By (\ref{assumption_redone}), Lemma~\ref{syz_deg_product_forms} implies that the above syzygy is the zero syzyzy: $L_i=(B/d)x_i$ for all $i=1,\ldots,n$. Plugging in  the first of the above equations and using the Euler relation for $f$ leads to $G\in \langle f\rangle$, a visible contradiction.
\end{proof}

\begin{Remark}\rm
	The format of the assumed inequality in the above proposition may leave some perplexity as to its validity.
	We now bear upon this issue.
	
For fixed $j=1,\ldots,m$, denote $\displaystyle D:=\sum_{i\neq j}d_i$ and $\displaystyle s:=\left\lfloor\frac{D}{m-1}\right\rfloor$. 
	
	(1) First, suppose $m\geq 3$. We have $2\leq d_1\leq\cdots\leq d_m$ such that $s-1\geq D -d_j+3$, and by the definition of the floor, $D=(m-1)s+u$, where $u\in\{0,\ldots,m-2\}$.
	
	So $D-u\geq (m-1)(D-d_j+4)$, which, from $d_{j-1}\geq\cdots\geq d_1\geq 2$ and $d_m\geq\cdots\geq d_{j+1}\geq d_j$, leads to
	
	$$(m-1)d_j\geq 2(m-2)+(m-2)(m-j)d_j+4(m-1)+u=(m-2)(m-j)d_j+6m-8+u.$$ Since, $m\geq 3$ and $u\geq 0$, we must have $j=m-1$ or $j=m$.
	Let us analyze these two cases separately.
	
	\begin{itemize}
		\item If $j=m-1$, then, with $d_m=d_{m-1}+v, v\geq 0$, we have $$d_{m-1}\geq (m-1)(d_1+\cdots+d_{m-2})+v(m-2)+4(m-1)+u.$$ For example, if $d_1=\cdots=d_{m-2}=2$, then the smallest values for $d_{m-1}$ and $d_m$ for which the hypotheses in this theorem hold true -- i.e., $u=v=0$, and therefore we can apply it -- are $d_m=d_{m-1}=2(m-1)m$.
		\item If $j=m$, then $$d_m\geq 4+\frac{(m-2)(d_1+\cdots+d_{m-1})+u}{m-1}.$$ If, for example, $d_1=\cdots=d_{m-1}=2$, with $u=0$, we have $d_m=2m$, the smallest possible.
	\end{itemize}
	
	Observe that the conclusion of Corollary~\ref{three_forms} fits the $j=m=3$ case, provided $d_3\geq 8$. This condition gives the required lower bound in Theorem \ref{unbalanced_degrees}:
	$$\mathfrak{d}(d_1,d_2):=\left\lfloor \frac{d_1+d_2}{2}\right\rfloor-1=\left\lfloor \frac{d_3}{2}\right\rfloor-1\geq 4-1=3=\underbrace{(d_1+d_2)-d_3}_{0}+3.$$ For $d_3\in \{5,6,7\}$ we cannot apply the theorem, but we can still apply the corollary to obtain the desired RTY-property.
	
	\medskip
	
	(2) If $m=2$, and if $d_1=d_2=2$ or $d_1=d_2=3$, then the conditions in Proposition~\ref{unbalanced_degrees} are not satisfied. Otherwise, $d_2\geq d_1-d_2+4$; so $j=2$ in the theorem gives the conclusion that $\{f_1,f_2\}$ satisfies the RTY-property. Observe that with this, we are also strengthening the conclusion of Corollary \ref{two_smooth_forms} to the case when $d_1\neq d_2$.
\end{Remark}

\bibliographystyle{amsalpha}


\end{document}